\documentclass[10pt,a4paper,reqno]{amsart}
\usepackage{vmargin}
\usepackage{caption}

\usepackage{amsmath,amsthm,amssymb,latexsym,amsfonts}
\usepackage{graphicx}
\usepackage{epsfig}
\usepackage[normalem]{ulem}
\usepackage{color}
\usepackage{mathrsfs}
\usepackage{enumitem}
\usepackage{comment}
\usepackage{mathtools}
\mathtoolsset{showonlyrefs}
\usepackage{hyperref}
\hypersetup{colorlinks,linkcolor={blue},citecolor={blue},urlcolor={blue}} 
\usepackage{accents}
\usepackage{esint}

\edef\restoreparindent{\parindent=\the\parindent\relax}
\usepackage{parskip}
\restoreparindent

\numberwithin{equation}{section}

\makeatletter
\def\bign#1{\mathclose{\hbox{$\left#1\vbox to8.5\p@{}\right.\n@space$}}\mathopen{}}
\makeatother

\makeatletter
\newcommand{\customlabel}[2]{%
   \protected@write \@auxout {}{\string \newlabel {#1}{{#2}{\thepage}{#2}{#1}{}} }%
   \hypertarget{#1}{}
}
\makeatother

\newtheorem{thm}{Theorem}[section]
\newtheorem{cor}[thm]{Corollary}
\newtheorem{prop}[thm]{Proposition}
\newtheorem{lem}[thm]{Lemma}

\theoremstyle{definition}

\theoremstyle{remark}
\newtheorem{rem}[thm]{Remark}

\newcommand{\la}{\langle}
\newcommand{\ra}{\rangle}
\newcommand{\xto}[1]{\xrightarrow{#1}}

\newcommand{\mres}{\mathbin{\vrule height 1ex depth 0pt width
0.1ex\vrule height 0.1ex depth 0pt width 1ex}}

\newcommand{\Cl}{\operatorname{Cl}}
\newcommand{\Int}{\operatorname{Int}}

\newcommand{\Div}{\operatorname{Div}}

\newcommand{\Span}{\operatorname{Span}}

\newcommand{\Co}{\operatorname{Co}}

\DeclareMathOperator*{\argmin}{arg\,min}
\newcommand{\Diam}{\operatorname{Diam}}

\newcommand{\Cone}{\operatorname{Cone}}

\let\epsilon\varepsilon

\let\emptyset\varnothing

\definecolor{Blue}{rgb}{0.024, 0.608, 0.839}

\title{A Strong Form of the Quantitative Wulff Inequality for Crystalline Norms}
\author{Kenneth DeMason}
\address{Department of Mathematics, The University of Texas at Austin,  2515 Speedway Stop C1200, Austin, Texas 78712-1202, USA}
\email{kdemason@utexas.edu}

\begin{document}
\begin{abstract} Quantitative stability for crystalline anisotropic perimeters, with control on the oscillation of the boundary with respect to the corresponding Wulff shape, is proven for $n\geq 3$. This extends a result of \cite{Neumayer2016} in $n=2$.
\end{abstract}
\maketitle
\renewcommand{\baselinestretch}{-5}\normalsize
\renewcommand{\baselinestretch}{1.0}\normalsize
\section{Introduction}
Wulff's construction, based on the classical theory of thermodynamics presented by Gibbs, provides a method for determining equilibrium shapes for crystals in media \cite{Wulff}. In this setting, the crystalline interfaces determine favorable orientations to minimize the so-called \textit{anisotropic perimeter} $\Phi$. Given a \textit{surface tension} $f:S^{n-1}\to [0,\infty)$, the anisotropic perimeter of a region $E\subset \mathbb{R}^n$ is defined as 
\begin{equation}
\Phi(E) = \int_{\partial^*E} f(\nu_E(x)) \ d\mathcal{H}^{n-1}(x).
\end{equation}
Here we use the standard notation and terminology, see \cite{Maggi}, where $E$ is a set of finite perimeter, $\partial^*E$ is its reduced boundary, and $\nu_E$ is the measure-theoretic outer unit normal.

This notion of perimeter arises naturally in crystallography when the $1$-homogeneous extension of $f$ to $\mathbb{R}^n$ is piecewise linear; in this case we call $f$ a \textit{crystalline} surface tension. The study of crystalline surface tensions is interesting from both a physical viewpoint due to its crystallographic connections \cite{GM,RW} and from a mathematical viewpoint \cite{Herring1951,McCann98,Taylor78}. 

From the perspective of statistical mechanics, identifying low-energy states of $\Phi$ is as important as identifying absolute minimizers, since they are interpreted as the most likely observable states of the system. Indeed, entropic contributions vary proportionally with temperature and thus always play a role in nature. It is therefore equally as interesting to know that the Wulff shape is the absolute minimizer of $\Phi$ at fixed volume as it is to describe, as precisely as possible, the shape of low-energy states.

The Wulff inequality, the anisotropic analogue of the isoperimetric inequality,
\begin{equation}\label{anisotropic Isop. ineq}
\Phi(E) \geq n|K|^{1/n}|E|^{(n-1)/n}
\end{equation}
characterizes the volume-constrained minimizers of the anisotropic perimeter as the corresponding \textit{Wulff shape} $K$, see \cite[Chapter 20]{Maggi}. Equality holds in \eqref{anisotropic Isop. ineq} if and only if $|E\Delta (rK+x)|=0$ for some $r>0$ and $x\in \mathbb{R}^n$. Due to this rigidity we can use the \textit{anisotropic deficit}
\begin{equation}\label{anisotropic deficit}
\delta_{\Phi}(E):= \frac{\Phi(E)}{n|K|^{1/n}|E|^{(n-1)/n}}-1
\end{equation}
to interpret the low-energy states as the regions $E$ for which $\delta_{\Phi}(E)$ is small. Indeed, by definition $\delta_{\Phi}(E)=0$ if and only if $E$ saturates \eqref{anisotropic Isop. ineq}. Note that $\delta_{\Phi}$ is scale invariant and thus detects deviations from $K$ due to shape alone.    

In \cite{FMP} the first identification of low-energy states of $\Phi$ in terms of control via $\delta_{\Phi}$ was obtained.  Here the authors show the existence of a universal constant $C(n)>0$ such that for any set of finite perimeter $E\subset \mathbb{R}^n$ with $0<|E|<\infty$,
\begin{equation}\label{sharp stability}
\alpha_{\Phi}(E)^2 \leq C\delta_{\Phi}(E).
\end{equation} 
In the above $\alpha_{\Phi}$ denotes the \textit{anisotropic Fraenkel asymmetry},
\begin{equation}\label{anisotropic asymmetry}
\alpha_{\Phi}(E) := \inf\left\lbrace \frac{|E\Delta (rK+x)|}{|E|} \ \bigg| \ |rK|=|E|, \ x\in \mathbb{R}^n\right\rbrace.
\end{equation}
Heuristically, this says that the low-energy states of $\Phi$ are close to the absolute minimizer in an $L^1$ sense. Due to the anisotropic nature of the problem \cite{FMP} relied on a Sobolev-Poincar\'{e} type trace inequality in conjunction with optimal transport methods as opposed to the symmetrization techniques available in the isotropic case, see the seminal result by \cite{FMP08}. More recently, \cite{CL} used a new technique known as the selection principle to re-prove the sharp quantitative isoperimetric inequality without appealing to symmetry; notably, this approach is applicable to the anisotropic setting.  

A particularly important class of surface tensions are the $\lambda$-elliptic surface tensions, for which the corresponding Wulff shape is $C^2$ and uniformly convex (with uniformity controlled by $\lambda>0$). In this case \cite{Neumayer2016}, using the ideas of \cite{CL} and \cite{FJ13}, improved the sharp stability \eqref{sharp stability} by showing that there exists $C(n,\lambda,\epsilon_{\Phi}, \|\nabla^2 f\|_{C^0(\partial K)})>0$
for every set of finite perimeter $E\subset \mathbb{R}^n$ with $0<|E|<\infty$,
\begin{equation}\label{strong sharp stability unif. elliptic}
\alpha_{\Phi}(E)^2+\beta_{\Phi}(E)^2 \leq C\delta_{\Phi}(E).
\end{equation}
Here $\beta_{\Phi}$ denotes the \textit{anisotropic oscillation index}
\begin{equation}\label{anisotropic oscillation}
\beta_{\Phi}(E):= \inf_{y\in \mathbb{R}^n}\left( \frac{1}{n|K|^{1/n}|E|^{(n-1)/n}}\int_{\partial^*E} f(\nu_E(x))-\frac{\la x-y, \nu_E(x)\ra}{f_*(x-y)} \ d\mathcal{H}^{n-1}(x)\right)^{1/2},
\end{equation}
where $f_*:\mathbb{R}^n\to [0,\infty)$ is the \textit{gauge function}, a $1$-homogeneous convex function related to the surface tension in a dual sense by
\begin{equation}
f_*(x) = \sup\{\la x,\nu \ra \ | \ f(\nu)\leq 1\}.
\end{equation}
The quantity $\beta_{\Phi}(E)$ was first introduced in the isotropic setting by \cite{FJ13} and controls the oscillation of the boundary with respect to a reference Wulff shape. Heuristically, \eqref{strong sharp stability unif. elliptic} says that low-energy states of $\Phi$ are close in an $H^1$ sense, see \cite[Proposition 1.9]{Neumayer2016}. 

We now turn to the crystalline setting, which as mentioned occurs when the $1$-homogeneous extension of $f$ to $\mathbb{R}^n$ is piecewise linear. Particularly, $\nabla^2 f =0$ a.e., in contrast to the $\lambda$-elliptic case, and the problem becomes degenerate elliptic. This degeneration also causes the Wulff shape $K$ to develop flat sides and become a polytope. Each facet $F_i$ of $K$ has a constant unit normal $\nu_i$ and lies at a distance ${\bf d_{\it{i}}} = f(\nu_i)$ from the origin. In this case $K$ is represented by
\begin{equation}
K = \bigcap_{i\in \mathscr{I}}\, \{x\in \mathbb{R}^n \ | \ \la x,\nu_i \ra < {\bf d_{\it{i}}}\}.
\end{equation}
In the above, $\mathscr{I}=\{1,...,N\}$ indexes the $N$ many facets of $K$. 

Recall as defined in \cite{FZ2019} that $E\subset \mathbb{R}^n$ is an $\epsilon$-minimizer for $\Phi$ if for every set of finite perimeter $F\subset \mathbb{R}^n$ with $|F|=|E|$ it holds
\begin{equation}
\Phi(E) \leq \Phi(F)+\epsilon |E\Delta F|.
\end{equation}
In particular when $|E|=|K|$ we have
\begin{equation}
\delta_{\Phi}(E) \leq \epsilon\left(\frac{|E\Delta K|}{\Phi(K)}\right) \leq \frac{2\epsilon}{n}.
\end{equation}
Thus $\epsilon$-minimizers are low-energy states of $\Phi$.   In fact, $\epsilon$-minimizers of $\Phi$ have been completely characterized, first for $n=2$ in \cite[Theorem 7]{FM2011} and later for $n\geq 3$ in \cite[Theorem 1.1]{FZ2019}, as polytopes with sides parallel to $K$. 

What remains then is to understand the remaining low-energy states which are not $\epsilon$-minimizers. Due to the anisotropic nature of the problem, little progress has been made in this direction from the lack of symmetrization techniques. The planar case was proven in \cite[Theorem 1.5]{Neumayer2016}, but the case for $n\geq 3$ remains open. The following theorem, the main result of this paper, proves this and therefore provides a natural conclusion to the above discussion.
\begin{thm}\label{Main Result} Let $f$ be a crystalline surface tension. There exists $C(n,K)>0$ such that for any $E\subset \mathbb{R}^n$ a set of finite perimeter with $0<|E|<\infty$,
$$\alpha_{\Phi}(E)^2+\beta_{\Phi}(E)^2 \leq C\delta_{\Phi}(E).$$
\end{thm}
Let us comment briefly on the proof strategy. The planar case of this theorem proven in \cite{Neumayer2016} relied on the characterization of $\epsilon$-minimizers in \cite{FM2011}. Analogously, our result relies on this characterization extended to $n\geq 3$ by \cite{FZ2019} which crucially uses the following theorem.
\setcounter{thm}{0}
\renewcommand*{\thethm}{\Alph{thm}}
\begin{thm}[{\cite[Theorem 1.4]{FZ2019}}]\label{Replacement} There exist $\sigma(n,K)>0$ and $\gamma(n,K)>0$ such that for any set of finite perimeter $E\subset \mathbb{R}^n$ with $|E|=|K|$ and $|E\Delta K|\leq \sigma$ there exists $K^{\bf a}$ a polytope parallel to $K$ with $|K^{\bf a}|=|K|$ and
$$\Phi(E)-\Phi(K^{\bf a})\geq \gamma |E\Delta K^{\bf a}|.$$
\end{thm}
\renewcommand*{\thethm}{\thesection.\arabic{thm}}
To expand on the notation above, we write $K^{\bf a}$ to denote the polytope obtained by perturbing $K$ via a vector ${\bf a}\in \mathbb{R}^N$ with $\|{\bf a}\|_{\ell^{\infty}}<1$ as follows:
\begin{equation}
K^{\bf a} = \bigcap_{i\in \mathscr{I}}\, \{x\in \mathbb{R}^n \ | \ \la x,\nu_i \ra < {\bf d_{\it{i}}}(1+{\bf a_{\it{i}}})\}.
\end{equation}
When ${\bf a}=0$ we have $K^{\bf a}=K$. The above procedure may produce a polytope having fewer sides than $K$, for example by perturbing the short side of a trapezoid to produce a triangle. This presents an issue as we frequently compare facets $F_i$ of $K$ with those of $K^{\bf a}$, necessitating they have the same unit normal. To ensure this occurs we need the technical condition that $K^{\bf a}$ is \textit{parallel} to $K$, defined by having the same number of sides or, equivalently, the same set of unit normals $\{\nu_i\}_{i=1}^N$. The collection of all such polytopes is denoted $\mathcal{C}_{\text{par}}(K)$.

In light of Theorem \ref{Replacement}, for sets of finite perimeter $E$ sufficiently close to $K$, we are able to essentially project them onto a polytope $K^{\bf a}$ also close to $K$ in terms of both having comparable volume and unit normals. This reduction allows us to leverage Theorem \ref{Replacement} and add control over $\beta_{\Phi}$.

Section \ref{Preliminaries} introduces some relevant results from \cite{Neumayer2016} and \cite{FZ2019}. Section \ref{Control Parallel Polytopes} proves Proposition \ref{Result for Parallel Polytopes}, the result for parallel polytopes. This is a delicate explicit computation involving estimates comparing the sizes of parallel facets, see Proposition \ref{Symmetric Difference Perturbed Facet Bound}. Section \ref{Proof of result} proves Theorem \ref{Main Result}. It first introduces some technical lemmas allowing us to control $\gamma_{\Phi}$, a quantity related to the oscillation index, in a strong way, via Proposition \ref{Vanishing Gradient}. It then proves the main result using the selection principle.

\subsection{Acknowledgments} The author wishes to thank Matias Delgadino, for reading preliminary drafts, several useful conversations, and constant encouragement; Francesco Maggi for useful suggestions; Marco Pozzetta for reading an early version; and Robin Neumayer for providing the author with this problem.

The author acknowledges the support of NSF-DMS RTG 1840314, the NSF Graduate Research Fellowship Program under Grant DGE 2137420, and UT Austin's Provost Graduate Excellence Fellowship.

\section{Summary of Relevant Results}\label{Preliminaries}
\subsection{Properties of anisotropic quantities} We first introduce $\gamma_{\Phi}$, which will allow us to more conveniently estimate $\beta_{\Phi}^2$. By applying the divergence theorem we can rewrite $\beta_{\Phi}(E)^2$ as
\begin{equation}\label{gamma}
\beta_{\Phi}(E)^2 = \frac{\Phi(E)-(n-1)\gamma_{\Phi}(E)}{n|K|^{1/n}|E|^{(n-1)/n}}, \ \ \ \gamma_{\Phi}(E):= \sup_{y\in \mathbb{R}^n}\int_{E}\frac{1}{f_*(x-y)} \ dx.
\end{equation}
\begin{rem}\label{Invariance} Note that $\alpha_{\Phi}(E)$, $\beta_{\Phi}(E)$, and $\gamma_{\Phi}(E)$ are all translation invariant. Moreover $\alpha_{\Phi}(E)$ and $\beta_{\Phi}(E)$ are scale invariant whereas $\gamma_{\Phi}(rE) = r^{n-1}\gamma_{\Phi}(E)$.
\end{rem}

The supremum defining $\gamma_{\Phi}(E)$ in \eqref{gamma} may not be attained at a unique point. We call any such point a \textit{center of} $E$, denoted $y_E$. The following lemma lets us estimate the norm of $y_E$. 
\begin{lem}[{\cite[Lemma 2.3]{Neumayer2016}}]\label{Center norm} For every $\epsilon>0$ there exists $\eta>0$ such that if $|E\Delta K|\leq \eta$ then $|y_E|<\epsilon$ for any center $y_E$ of $E$.
\end{lem}

We summarize the important properties of $\beta_{\Phi}$ and $\gamma_{\Phi}$ here.
\begin{prop}[{\cite[Prop. 2.1]{Neumayer2016}}]\label{Properties} The function $\gamma_{\Phi}$ is H\"{o}lder continuous with respect to $L^1$ convergence, that is
$$|\gamma_{\Phi}(E)-\gamma_{\Phi}(F)| \leq \frac{n|K|}{(n-1)}|E\Delta F|^{(n-1)/n}.$$
Furthermore, for any sequence $\{E_j\}_{j=1}^{\infty}$ such that $E_j$ converges to $E$ in $L^1$ and $\{f^j\}_{j=1}^{\infty}$ a sequence of surface tensions converging locally uniformly to $f$, we have
\begin{enumerate}[label = \roman*)]
\item $\gamma_{\Phi_{\bullet}}$ is continuous;
$$\lim_{j\to \infty} \gamma_{\Phi_j}(E_j) = \gamma_{\Phi_j}(E).$$
\item $\beta_{\Phi_{\bullet}}$ and $\Phi$ are lower semi-continuous,
$$\beta_{\Phi}(E)\leq \liminf_{j\to \infty} \beta_{\Phi_j}(E_j), \ \ \ \Phi(E)\leq \liminf_{j\to \infty} \Phi(E_j).$$
Additionally if $\Phi(E_j)\to \Phi(E)$ then $\beta_{\Phi}(E_j)\to \beta_{\Phi}(E)$ too.
\end{enumerate}
\end{prop}

\subsection{Technical results} The following results come from \cite{Neumayer2016} and \cite{FZ2019}. 

In the proof of Theorem \ref{Main Result} we will need to apply the selection principle. To do this we will assume Theorem \ref{Main Result} is false and generate a sequence $\{E_j\}_{j=1}^{\infty}$. This sequence ends up being too generic to use, and we must replace it with a new sequence $\{F_j\}_{j=1}^{\infty}$ with more regularity, which will be obtained as minimizers of a certain functional. The first step in this is the following lemma, which allows us to restrict our attention to those sets $E$ contained in some ball. 
\begin{lem}[{\cite[Lemma 3.1]{Neumayer2016}}]\label{Ball containment} There exist constants $R_0(n,K)>0$ and $C(n,K)>0$  such that for any set of finite perimeter $E$ with $|E|=|K|$ we can find $E'\subset B_{R_0}$ such that $|E'|=|K|$ and
$$\beta_{\Phi}(E)^2 \leq \beta_{\Phi}(E')^2+C\delta_{\Phi}(E), \ \ \ \delta_{\Phi}(E')\leq C\delta_{\Phi}(E).$$
\end{lem}

To replace the $E_j$ we construct $F_j$ as minimizers of the functional
$$Q(E) = \Phi(E) + \epsilon_{\Phi}|K||\beta_{\Phi}(E)^2-\epsilon^2| + \Lambda\big||E|-|K|\big|$$
for $0<\epsilon<1$ and $\Lambda>0$; in practice we take $\epsilon = \beta_{\phi}(E_j)$. Here, $\epsilon_{\Phi}=m_{\Phi}/M_{\Phi}$ the so-called eccentricity, where 
\begin{equation}
M_{\Phi}=\sup_{\nu \in S^{n-1}} f(\nu) \ \ \ \text{and} \ \ \ m_{\Phi} = \inf_{\nu \in S^{n-1}} f(\nu).
\end{equation}
\indent
To explain the relevance of $Q$, the latter two terms are penalty terms to control the oscillation and volume. Since these penalizations should be small, minimizers of $Q$ are almost minimizers of $\Phi$, and thus should be comparable to $K$. Indeed by the rigidity of \eqref{anisotropic Isop. ineq} we know that volume-constrained minimizers of $\Phi$ are the Wulff shape $K$ up to translation. We can reformulate this without the volume constraint as
\begin{lem}[{\cite[Lemma 3.5]{Neumayer2016}}]\label{Unique Min} Let $R_0>\Diam(K)$ and $\Lambda>n$. Then up to translation $K$ is the unique minimizer of the functional
$$\Phi(E) + \Lambda \big||E|-|K|\big|$$
over $E\subset B_{R_0}$.
\end{lem}
Note the similarity between the functional in Lemma \ref{Unique Min} and $Q$. The following lemma guarantees that minimizers of $Q$ exist.
\begin{lem}[{\cite[Lemma 3.2]{Neumayer2016}}]\label{Minimize Q} A minimizer for the problem
$$\min\{Q(E) \ | \ E\subset B_{R_0}\}$$
exists for $\Lambda>4n$ and $\epsilon>0$ sufficiently small. Moreover, any minimizer $F$ satisfies
\begin{equation}\label{Minimize Q Properties}
|F| \geq \frac{|K|}{2}, \ \ \ \ \Phi(F) \leq 2n|K|=2\Phi(K).
\end{equation}
\end{lem}
The former condition in \eqref{Minimize Q Properties} tells us that $F$ does not degenerate, whereas the latter gives us a nontrivial way to compare $\Phi(F)$ with $\Phi(K)$. This will be crucial in producing a contradiction.

We say that $E$ satisfies the following \textit{uniform density estimate} if there exists $r_0(n,K)>0$ and $0<c_0(n)<1/2$ such that for any $x\in \partial^*E$ and $r<r_0$ it holds
\begin{equation}\label{Uniform density estimate}
c_0\epsilon_{\Phi}^n\omega_n r^n \leq |B_r(x)\cap E| \leq \left(1-c_0\epsilon_{\Phi}^n\right)\omega_n r^n.
\end{equation}
Typically $L^1$ convergence is not enough to control Hausdorff convergence. However, combined with uniform density estimates one can obtain Hausdorff convergence of the boundaries, which in turn controls the Hausdorff distance. This is a standard technique; see for example \eqref{hausdorff closeness E} of Lemma \ref{Closeness}. Indeed, notably uniform density estimates help rule out cases like having long, thin tentacles which can extend arbitrarily far.

Importantly, minimizers of $Q(E)$ satisfy the uniform density estimates.
\begin{lem}[{\cite[Lemma 3.3]{Neumayer2016}}]\label{Minimize Q Uniform} Any minimizer $F$ of $Q(E)$ among $E\subset B_{R_0}$  satisfies the uniform density estimates \eqref{Uniform density estimate}. 
\end{lem}
\begin{rem} In \eqref{Uniform density estimate} we require that $r_0$ depends on $n$ and $K$. In actuality, in Lemma \ref{Minimize Q Uniform} it also depends on $\Lambda$. However, we will take $\Lambda$ to be some fixed constant $\Lambda>4n$, so this will not matter.
\end{rem}
Also, the uniform density estimates allow us to upgrade the H\"{o}lder continuity from Proposition \ref{Properties} to a kind of Lipschitz continuity; see Proposition \ref{Vanishing Gradient}. In order to achieve this, we need the following lemma from \cite{FZ2019}. 
\begin{lem}[{\cite[Lemma 2.3]{FZ2019}}]\label{Modulus continuity} There exists a modulus of continuity $\omega:(0,1)\to (0,\infty)$ such that if $|{\bf a}|+|{\bf a'}|<1$ then
$$|\Phi(K^{{\bf a}})-\Phi(K^{{\bf a'}})| \leq \omega(|{\bf a}|+|{\bf a'}|)|{\bf a}-{\bf a'}|.$$
\end{lem}
Finally we note the following useful property of $\mathcal{C}_{\text{par}}(K)$.
\begin{rem}\label{Weak Equivalence} As $\mathcal{C}_{\text{par}}(K)$ is parametrized by ${\bf a}\in \mathbb{R}^N$ it is finite dimensional. We expect the metrics induced by the $L^1$ distance, Hausdorff distance, and $\|{\bf a}\|_{\ell^{\infty}}$ are all equivalent up to a constant $C(n,K)>0$; indeed, this is the content of Lemma \ref{Equivalence} below. Explicit computations show that, in fact, these constants can be chosen to depend only on $n, N, \epsilon_{\Phi}, \csc(\theta_{ij})$ where $\cos(\theta_{ij})=-\la \nu_i,\nu_j\ra$. As in \cite{FMP}, it is likely that the dependence on $\epsilon_{\Phi}$ can be removed by an application of John's lemma.

One may hope to use Theorem \ref{Main Result} to re-prove the uniformly elliptic result (see \cite[Theorem 1.3]{Neumayer2016}) by approximating a smooth Wulff shape $K$ with polytopes $K_j$, thereby removing the dependence on $\|\nabla^2 f\|_{C^0(\partial K)}$ of the universal constant. However, this necessarily means both $\csc(\theta_{ij})\to \infty$ and $N\to \infty$. For this reason, we do not include the extra work to explicitly compute these constants in terms of these parameters.
\end{rem}

\section{Parallel Deformations of Polytopes}\label{Control Parallel Polytopes}

\subsection{Comparison of parallel facets} Given $K^{\bf a},K^{\bf a'}\in \mathcal{C}_{\text{par}}(K)$, since they are parallel they have the same set of unit normals $\{\nu_i\}_{i\in \mathscr{I}}$. Let $F_i^{\bf a}, F_i^{\bf a'}$ be the facets of $K^{\bf a}, K^{\bf a'}$ respectively with outer unit normal $\nu_i$. Formally, $F_i^{\bf a} = \partial K^{\bf a} \cap \Sigma_{\nu_i}^{1+{\bf a_{\it{i}}}}$, where
\begin{equation}
\Sigma_{\nu_i}^r := \{\la x,\nu_i \ra = r{\bf d_{\it{i}}}\}.
\end{equation}
Notice that $\Sigma_{\nu_i}^{1+{\bf a_{\it{i}}}}$ is the supporting hyperplane of $K^{\bf a}$ in the direction $\nu_i \in S^{n-1}$. 

One can ask how $F_i^{\bf a}$ and $F_i^{\bf a'}$ are related in terms of, say, $\|{\bf a}-{\bf a'}\|_{\ell^{\infty}}$. Indeed as in Lemma \ref{Equivalence} we see that if $\{K^{\bf a_j}\}_{j=1}^{\infty}\subset \mathcal{C}_{\text{par}}(K)$ is such that ${\bf a_j}\xto{\ell^{\infty}} {\bf a}$ then $d_H(K^{\bf a_j},K^{\bf a})\to 0$. This is a strong condition, and suggests that $F_i^{\bf a}$ and $F_i^{\bf a_j}$ can be compared for sufficiently large $j$. In fact, we will show that, up to a suitable translation, $\mathcal{H}^{n-1}(F_i^{\bf a_j}\Delta F_i^{\bf a})\to 0$ and quantify the rate of convergence. More precisely we prove:
\begin{prop}\label{Symmetric Difference Perturbed Facet Bound} There exists $C(n,K)>0$ such that for any $K^{\bf a},K^{\bf a'}\in \mathcal{C}_{\text{par}}(K)$ and any $i\in \mathscr{I}$,
\begin{equation}
\mathcal{H}^{n-1}(F_i^{\bf a}\Delta (F_i^{\bf a'}+u_i))\leq CM_{\Phi}^{n-1}\|{\bf a}-{\bf a'}\|_{\ell^{\infty}}
\end{equation}
where $u_i = {\bf d_{\it{i}}}({\bf a_{\it{i}}}-{\bf a_{\it{i}}'})\nu_i$.
\end{prop}
In particular, applying Proposition \ref{Symmetric Difference Perturbed Facet Bound} with $K$ and $K^{\bf a}$ immediately yields
\begin{cor}[Asymptotics of Parallel Facet Areas]\label{Asymptotics of Parallel Facet Areas} There exists  $C(n,K)>0$ such that for any $K^{\bf a}\in \mathcal{C}_{\text{par}}(K)$,
\begin{equation}\label{Facet measure diff O}
|\mathcal{H}^{n-1}(F_i^{\bf a})-\mathcal{H}^{n-1}(F_i)| \leq CM_{\Phi}^{n-1}\|{\bf a}\|_{\ell^{\infty}}
\end{equation}
that is $|\mathcal{H}^{n-1}(F_i^{\bf a})-\mathcal{H}^{n-1}(F_i)| = O(\|{\bf a}\|_{\ell^{\infty}}).$
\end{cor}

We start by showing that the class $\mathcal{C}_{\text{par}}(K)$ is non-empty for $\|{\bf a}\|_{\ell^{\infty}}$ sufficiently small.
\begin{lem}\label{Parallel Polytopes Exist} There exists $a_0(n,K)>0$ such that if $\|{\bf a}\|_{\ell^{\infty}}\leq a_0$ and $K^{\bf a}$ is the Wulff shape associated to the gauge function 
\begin{equation}
f_*^{\bf a}(x) = \max_{i\in \mathscr{I}}\frac{\la x,\nu_i \ra}{{\bf d_{\it{i}}}(1+{\bf a_{\it{i}}})},
\end{equation}
then $K^{\bf a}$ is parallel to $K$.
\end{lem} 
\begin{proof}
The idea is to show we can find a neighborhood $U$ of ${\bf 0}$ such that for any ${\bf a}\in U$ and $i\in \mathscr{I}$ we have $\mathcal{H}^{n-1}(F_i^{\bf a})>0$. To do this we first prove the map ${\bf a}\mapsto \mathcal{H}^{n-1}(F_i^{\bf a})$ for fixed $i\in \mathscr{I}$ is continuous. Note that Corollary \ref{Asymptotics of Parallel Facet Areas} determines a precise modulus of continuity, but this requires the assumption that $K^{\bf a}$ is parallel to $K$.

Fix $i\in \mathscr{I}$ and consider ${\bf a}$ such that $|{\bf a_{\it{i}}}|<1$ and ${\bf a_{\it{j}}}=0$ for $j\neq i$. Throughout we denote $F_j^{\bf a}$ instead as $F_j^{\bf a_{\it{i}}}$ to emphasize that ${\bf a}$ depends only on its $i$-th component ${\bf a_{\it{i}}}$. First let $j\neq i$ and observe that $F_j^{\bf a_{\it{i}}}\subset F_j^{\bf a_{\it{i}}'}$ for ${\bf a_{\it{i}}}\leq {\bf a_{\it{i}}'}$. Since $|{\bf a_{\it{i}}}|<1$ we see that $F_j^{1}$ is maximal, in the sense that $F_j^{\bf a_{\it{i}}}\subset F_j^1$ for all $|{\bf a_{\it{i}}}|<1$. Furthermore, it follows that 
\begin{equation}
F_j^{\bf a_{\it{i}}} = F_j^{1} \cap \{\la x,\nu_i \ra \leq {\bf d_{\it{i}}}(1+{\bf a_{\it{i}}})\}.
\end{equation} 
In particular, for $s\in [{\bf a_{\it{i}}},{\bf a_{\it{i}}'}]$ the slices $(F_j^{\bf a_{\it{i}}'}\setminus F_j^{\bf a_{\it{i}}})\cap \Sigma_{\nu_i}^{1+s}$ and $F_j^{1} \cap \Sigma_{\nu_i}^{1+s}$ coincide. Consequently by Fubini's theorem,
\begin{align}
\mathcal{H}^{n-1}(F_j^{\bf a_{\it{i}}'})-\mathcal{H}^{n-1}(F_j^{\bf a_{\it{i}}}) &= \mathcal{H}^{n-1}(F_j^{\bf a_{\it{i}}'}\setminus F_j^{\bf a_{\it{i}}})\\
&= \int_{\bf a_{\it{i}}}^{{\bf a_{\it{i}}'}} \mathcal{H}^{n-2}((F_j^{\bf a_{\it{i}}'}\setminus F_j^{\bf a_{\it{i}}})\cap \Sigma_{\nu_i}^{1+s}) \ ds = \int_{\bf a_{\it{i}}}^{{\bf a_{\it{i}}'}} \mathcal{H}^{n-2}(F_j^{1}\cap \Sigma_{\nu_i}^{1+s}) \ ds
\end{align}
and $\mathcal{H}^{n-1}(F_j^{\bf a_{\it{i}}})$ is absolutely continuous in ${\bf a_{\it{i}}}$ by choosing ${\bf a_{\it{i}}'}=0$ and reversing roles if necessary.

It remains to show that $\mathcal{H}^{n-1}(F_i^{\bf a_{\it{i}}})$ is continuous. By \cite[Lemma 5.1.1]{Schneider} we have
\begin{equation}\label{continuity of Fi}
|K^{\bf a_{\it{i}}}| = \frac{1}{n}\sum_{j\neq i} {\bf d_{\it{j}}}\mathcal{H}^{n-1}(F_j^{\bf a_{\it{i}}}) + \frac{{\bf d_{\it{i}}}(1+{\bf a_{\it{i}}})}{n}\mathcal{H}^{n-1}(F_i^{\bf a_{\it{i}}}).
\end{equation}
By the same slicing argument above we see that $|K^{\bf a_{\it{i}}}|$ is continuous. Hence by \eqref{continuity of Fi}, so too is $\mathcal{H}^{n-1}(F_i^{\bf a_{\it{i}}})$. We may now choose $a_0$ so that $\mathcal{H}^{n-1}(F_j^{{\bf a_{\it{i}}}})>\mathcal{H}^{n-1}(F_j)/2$ for all $j\in \mathscr{I}$.

The previous argument works whenever we perturb one of the facets. Now, given arbitrary ${\bf a}$ we construct ${\bf a_1},...,{\bf a_N}$ so that the $j$-th components satisfy ${\bf a_{i,\it{j}}} = {\bf a_{\it{j}}}$ for $j\leq i$ and ${\bf a_{i,\it{j}}}=0$ for $j>i$. Notice that ${\bf a_N}={\bf a}$, and set ${\bf a_0}={\bf 0}$. By the above, if each ${\bf a_{\it{i}}}$ is sufficiently small then 
\begin{equation}
\mathcal{H}^{n-1}(F_j^{\bf a}) = \mathcal{H}^{n-1}(F_j^{\bf a_N}) \geq \frac{1}{2}\mathcal{H}^{n-1}(F_j^{\bf a_{N-1}})\geq \frac{1}{2^N}\mathcal{H}^{n-1}(F_j^{\bf a_0})=\frac{1}{2^N}\mathcal{H}^{n-1}(F_j)>0.
\end{equation} 
In principle the choice of ${\bf a_{\it{i}}}$ depends on that of ${\bf a_{\it{1}}},...,{\bf a_{\it{i}-1}}$, but since this holds for any $i\in \mathscr{I}$, they all depend on ${\bf a_{\it{1}}}$, which in turn depends on $n$ and $K$. 
\end{proof}
We will also need to determine a modulus of continuity for the maps ${\bf a}\mapsto |K^{\bf a}\Delta K|$ and ${\bf a}\mapsto d_H(K^{\bf a},K)$.
\begin{lem}\label{Equivalence}
There exists $C(n,K)>0$ such that for any $K^{\bf a},K^{\bf a'}\in \mathcal{C}_{\text{par}}(K)$ we have
\begin{equation}
\|{\bf a}-{\bf a'}\|_{\ell^{\infty}}\leq C|K^{\bf a}\Delta K^{\bf a'}|\leq C^2d_H(K^{\bf a},K^{\bf a'})\leq C^3\|{\bf a}-{\bf a'}\|_{\ell^{\infty}}.
\end{equation}\label{equivalent metrics}
In particular, the metrics induced by $\|\cdot\|_{\ell^{\infty}}$, $d_H$, and $L^1$ are all strongly equivalent. 
\end{lem}
\begin{proof} We first show there exists a $C(n,K)$ such that $d_H(K^{\bf a},K^{\bf a'})\leq C\|{\bf a}-{\bf a'}\|_{\ell^{\infty}}$. From \cite[Theorem 1.8.11]{Schneider} the Hausdorff distance between convex sets $K$ and $K'$ with surface tensions $f$ and $f'$ respectively can be computed as
\begin{equation}\label{Hausdorff ell infinity}
d_H(K,K')=\|f-f'\|_{L^{\infty}(S^{n-1})}.
\end{equation}
Fix $K^{\bf a}$ and $K^{\bf a'}$ (not necessarily parallel). Then without loss of generality there exists $\nu \in S^{n-1}$ such that 
\begin{equation}
d_H(K^{\bf a},K^{\bf a'}) = f^{\bf a}(\nu)-f^{\bf a'}(\nu) = \sup_{x\in K^{\bf a}}\la x,\nu\ra - \sup_{x\in K^{\bf a'}}\la x,\nu\ra.
\end{equation}
Let $x^{\bf a}\in \partial K^{\bf a}$ attain the first supremum. Recall that $x\in \partial K^{\bf a'}$ if and only if $f_*^{\bf a'}(x)=1$. By 1-homogeneity it follows that $x^{\bf a}/f_*^{\bf a'}(x^{\bf a})\in \partial K^{\bf a}$, and is a valid candidate in the latter supremum. Thus,
\begin{equation}
d_H(K^{\bf a},K^{\bf a'}) \leq \la x^{\bf a},\nu \ra - \left\la \frac{x^{\bf a}}{f_*^{\bf a'}(x^{\bf a})},\nu \right\ra = \left\la \left[\frac{f_*^{\bf a'}(x^{\bf a})-1}{f_*^{\bf a'}(x^{\bf a})}\right]x^{\bf a},\nu \right\ra \leq \bigg|\frac{f_*^{\bf a'}(x^{\bf a})-1}{f_*^{\bf a'}(x^{\bf a}/|x^{\bf a}|)} \bigg|.
\end{equation}
Appealing to the facts that $f_*^{\bf a}(x^{\bf a})=1$, $M_{\Phi}^{\bf a'} = \sup_{x \in S^{n-1}} 1/f_*^{\bf a'}(x)$, and $K^{\bf a}\subset B_{M_{\Phi}^{\bf a}}(0)$ we have
\begin{align}
d_H(K^{\bf a},K^{\bf a'}) &\leq M_{\Phi}^{\bf a'}|f_*^{\bf a'}(x^{\bf a})-f_*^{\bf a}(x^{\bf a})| \\
&\leq |x^{\bf a}|M_{\Phi}^{\bf a'}\left|f_*^{\bf a'}\left(\frac{x^{\bf a}}{|x^{\bf a}|}\right) - f_*^{\bf a}\left(\frac{x^{\bf a}}{|x^{\bf a}|}\right)\right| \leq M_{\Phi}^{\bf a}M_{\Phi}^{\bf a'}\|f_*^{\bf a'}-f_*^{\bf a}\|_{L^{\infty}(S^{n-1})}.
\end{align}
Finally observe that $K^{\bf a}\subset K^{\bf 1}$, where ${\bf 1}=(1,...,1)$. Letting $R(n,K)>0$ denote the smallest $R>0$ such that $K^{\bf 1}\subset B_R(0)$, we finally have
\begin{equation}
d_H(K^{\bf a},K^{\bf a'}) \leq R^2 \|f_*^{\bf a'}-f_*^{\bf a}\|_{L^{\infty}(S^{n-1})}.
\end{equation}
It now suffices to show there exists $C(n,K)>0$ such that $\|f_*^{\bf a'}-f_*^{\bf a}\|_{L^{\infty}(S^{n-1})}\leq C\|{\bf a}-{\bf a'}\|_{\ell^{\infty}}$. As before, without loss of generality we may find $\nu \in S^{n-1}$ such that
\begin{equation}
\|f_*^{\bf a'}-f_*^{\bf a}\|_{L^{\infty}(S^{n-1})} = f_*^{\bf a'}(\nu)- f_*^{\bf a}(\nu) = \max_{i\in \mathscr{I}} \frac{\la \nu, \nu_i \ra}{{\bf d_{\it{i}}}(1+{\bf a_{\it{i}}'})}-\max_{i\in \mathscr{I}} \frac{\la \nu, \nu_i \ra}{{\bf d_{\it{i}}}(1+{\bf a_{\it{i}}})}.
\end{equation}
Letting $\nu_i$ attain the first maximum, we have
\begin{equation}
\|f_*^{\bf a'}-f_*^{\bf a}\|_{L^{\infty}(S^{n-1})} \leq  \frac{\la \nu, \nu_i \ra}{{\bf d_{\it{i}}}(1+{\bf a_{\it{i}}'})}-\frac{\la \nu, \nu_i \ra}{{\bf d_{\it{i}}}(1+{\bf a_{\it{i}}})} = \left[\frac{{\bf a_{\it{i}}}-{\bf a_{\it{i}}'}}{{\bf d_{\it{i}}}(1+{\bf a_{\it{i}}'})(1+{\bf a_{\it{i}}})}\right]\la \nu,\nu_i \ra.
\end{equation}
In particular, choosing $a_0\leq 1/2$ we have that ${\bf a_{\it{i}}},{\bf a_{\it{i}}'}\geq -1/2$ and
\begin{equation}
\|f_*^{\bf a'}-f_*^{\bf a}\|_{L^{\infty}(S^{n-1})} \leq \frac{4}{{\bf d_{\it{i}}}}|{\bf a_{\it{i}}}-{\bf a_{\it{i}}'}| \leq \frac{4}{m_{\Phi}}\|{\bf a}-{\bf a'}\|_{\ell^{\infty}}
\end{equation}
as $m_{\Phi} = \inf_{\nu\in S^{n-1}}f(\nu)$, and $f(\nu_i) = {\bf d_{\it{i}}}$ by definition. 

Next we show there exists $C(n,K)>0$ such that $\|{\bf a}-{\bf a'}\|_{\ell^{\infty}}\leq C|K^{\bf a}\Delta K^{\bf a'}|$. Fix $K^{\bf a}$ and $K^{\bf a'}$, and consider $K^{\bf a_s}$ where ${\bf a_s} = (1-s){\bf a}+s{\bf a'}$. Let 
\begin{equation}
E_j = \bigcup_{s\in [0,1]} F_j^{\bf a_s}\cap (K^{\bf a}\Delta K^{\bf a'})
\end{equation}
and observe the $E_j$ partition $K^{\bf a}\Delta K^{\bf a'}$ (away from sets of measure zero), having parallel facets $F_j^{\bf a}\cap (K^{\bf a}\Delta K^{\bf a'})$ and $F_j^{\bf a'}\cap (K^{\bf a}\Delta K^{\bf a'})$. Next by slicing,
\begin{equation}
|K^{\bf a}\Delta K^{\bf a'}|= \sum_{j\in \mathscr{I}}|E_j| = \sum_{j\in \mathscr{I}}\bigg|\int_{{\bf a_{\it{j}}}}^{\bf a_{\it{j}}'} \mathcal{H}^{n-1}(E_j \cap \Sigma_{\nu_j}^{1+s}) \ ds\bigg| \geq \sum_{j\in \mathscr{I}} c_j|{\bf a_{\it{j}}}-{\bf a_{\it{j}}'}| \geq c\|{\bf a}-{\bf a'}\|_{\ell^{\infty}}
\end{equation}
where $c_j = \min_{s\in [0,1]}\{\mathcal{H}^{n-1}(E_j \cap \Sigma_{\nu_j}^{1+s})\}$ and $c = \min_{j\in \mathscr{I}}\{c_j\}$. In principle these depend on ${\bf a}$ and ${\bf a'}$, but once more we can take a minimum over $[-a_0,a_0]^N$ to remove this dependence. 

Finally we show $|K^{\bf a}\Delta K^{\bf a'}|\leq Cd_H(K^{\bf a},K^{\bf a'})$. The main theorem of \cite{Groemer}, part i) states that: If $D=\max\{\Diam(K^{\bf a}), \Diam(K^{\bf a'})\}$, then
\begin{equation}
|K^{\bf a}\Delta K^{\bf a'}| \leq C\,d_H(K^{\bf a},K^{\bf a'}) , \ \ \  C = \frac{2\omega_n}{2^{1/n}-1}\left(\frac{D}{2}\right)^{n-1}.
\end{equation}
But, with $R$ defined earlier, we have that $\Diam(K^{\bf a})\leq 2R$ for any $K^{\bf a}$, and so we may choose $C=2\omega_nR^{n-1}/(2^{1/n}-1)$. 
\end{proof}

We now consider parallel perturbations in one direction. That is, we suppose that $K^{\bf a},K^{\bf a'}\in \mathcal{C}_{\text{par}}(K)$ are such that there exists $i\in \mathscr{I}$ where for all $j\in \mathscr{I}$ with $j\neq i$ we have ${\bf a_{\it{j}}}={\bf a_{\it{j}}'}$. 
We will call this a \textit{parallel perturbation in the direction} $\nu_i$.

We first estimate $\mathcal{H}^{n-1}(F_i^{\bf a}\Delta F_i^{\bf a'})$, up to suitable translation, in terms of the $L^1$ norm of the other facets. 
\begin{lem}[Bound on parallel facets in the perturbed direction]\label{Symmetric Difference Perturbed Parallel Facet Bound} Let $K^{\bf a},K^{\bf a'}\in \mathcal{C}_{\text{par}}(K)$ be a parallel perturbation in the direction $\nu_i$. Then,
\begin{equation}\label{Sym Diff Perturbed Parallel Facet Bound}
\mathcal{H}^{n-1}(F_i^{\bf a} \Delta (F_i^{\bf a'}+u_i))\leq \sum_{j\neq i}\mathcal{H}^{n-1}(F_j^{\bf a}\Delta F_j^{\bf a'})
\end{equation}
where $u_i = {\bf d_{\it{i}}}({\bf a_{\it{i}}}-{\bf a_{\it{i}}'})\nu_i$.
\end{lem}
\begin{proof} $\,$

\noindent
\textit{Step 1:} We first note that if $E,F\subset \mathbb{R}^n$ are Borel then,
\begin{equation}\label{Sym Diff cont func}
|E\Delta F| = \sup_{\varphi\in C^0(\mathbb{R}^n), |\varphi|\leq 1}\left[\int_E \varphi(x) \ dx - \int_F \varphi(x) \ dx\right].
\end{equation}
By the Riesz Representation theorem, $(C^0(\mathbb{R}^n),\|\cdot\|_{L^{\infty}})^* \simeq (\mathcal{M}(\mathbb{R}),\|\cdot\|_{TV})$ where $\mathcal{M}(\mathbb{R}^n)$ is the set of (signed) Borel measures on $\mathbb{R}^n$. Since $|E\Delta F| = \|\mu_E -\mu_F\|_{TV}$ where $\mu_E =\mathcal{H}^n\mres E$, \eqref{Sym Diff cont func} is then a consequence of computing $\|\mu_E -\mu_F\|_{TV}$ via Riesz Representation as above. We will actually apply this on $\mathbb{R}^{n-1}$ instead of $\mathbb{R}^n$, but for readability it is easier to state as above.

\noindent
\textit{Step 2:} We next simplify the computation by showing that
\begin{equation}\label{Facet Sym Diff Supremum}
\mathcal{H}^{n-1}(F_i^{\bf a} \Delta (F_i^{\bf a'}+u_i)) = \sup_{\varphi\in C_{\nu_i}^0(\mathbb{R}^n), |\varphi|\leq 1}\left[\int_{F_i^{\bf a}} \varphi(x) \ d\mathcal{H}^{n-1}(x) - \int_{F_i^{\bf a'}} \varphi(x) \ d\mathcal{H}^{n-1}(x)\right]
\end{equation}
where 
$$C_{\nu_i}^0(\mathbb{R}^n) = \{\varphi \in C^0(\mathbb{R}^n) \ | \ \varphi(x) = \varphi(x+t\nu_i) \ \text{for all} \ t\in\mathbb{R}\}.$$
First notice that
\begin{align}
&\sup_{\varphi\in C_{\nu_i}^0(\mathbb{R}^n), |\varphi|\leq 1}\left[\int_{F_i^{\bf a}} \varphi(x) \ d\mathcal{H}^{n-1}(x) - \int_{F_i^{\bf a'}} \varphi(x) \ d\mathcal{H}^{n-1}(x)\right] \nonumber \\
& \ \ \ \ \ \ \ \ \ \ = \sup_{\phi\in C^0(\mathbb{R}^{n-1}), |\phi|\leq 1}\left[\int_{L(F_i^{\bf a})} \phi(y) \ dy - \int_{L(F_i^{\bf a'}+u_i)} \phi(y) \ dy\right] \label{Invariance conversion}
\end{align}
where $\Sigma_{\nu_i}^{1+{\bf a_{\it{i}}}}:=\{\la x,\nu_i \ra = {\bf d_{\it{i}}}(1+{\bf a_{\it{i}}})\}$ and $L:\Sigma_{\nu_i}^{1+{\bf a_i}}\to \mathbb{R}^{n-1}$ is a fixed isometry. This proves \eqref{Facet Sym Diff Supremum} since
\begin{align*}
\sup_{\phi\in C^0(\mathbb{R}^{n-1}), |\phi|\leq 1}\left[\int_{L(F_i^{\bf a})} \phi(y) \ dy - \int_{L(F_i^{\bf a'}+u_i)} \phi(y) \ dy\right] &= \mathcal{L}^{n-1}(L(F_i^{\bf a})\Delta L(F_i^{\bf a'}+u_i))\\
&= \mathcal{L}^{n-1}(L(F_i^{\bf a}\Delta (F_i^{\bf a'}+u_i)))\\
&= \mathcal{H}^{n-1}(F_i^{\bf a}\Delta (F_i^{\bf a'}+u_i))
\end{align*}
by Step 1 and the fact that $L$ is an isometry. To show \eqref{Invariance conversion}, given $\varphi\in C_{\nu_i}^0(\mathbb{R}^n)$ with $|\varphi|\leq 1$,
$$\int_{F_i^{\bf a'}} \varphi(x) \ d\mathcal{H}^{n-1}(x) = \int_{F_i^{\bf a'}} \varphi(x-u_i) \ d\mathcal{H}^{n-1}(x) = \int_{F_i^{\bf a'}+u_i} \varphi(x) \ d\mathcal{H}^{n-1}(x),$$
owing to the invariance of $\varphi$ along $\nu_i$. Then, setting $\phi(y) = \varphi(L^{-1}(y))$ we have
$$\int_{F_i^{\bf a}} \varphi(x) \ d\mathcal{H}^{n-1}(x)-\int_{F_i^{\bf a'}+u_i} \varphi(x) \ d\mathcal{H}^{n-1}(x) = \int_{L(F_i^{\bf a})}\phi(y) \ dy-\int_{L(F_i^{\bf a'}+u_i)}\phi(y) \ dy.$$
Since $\phi \in C^0(\mathbb{R}^{n-1})$ and $|\phi|\leq 1$, we necessarily get
\begin{align*}
&\sup_{\varphi\in C_{\nu_i}^0(\mathbb{R}^n), |\varphi|\leq 1}\left[\int_{F_i^{\bf a}} \varphi(x) \ d\mathcal{H}^{n-1}(x) - \int_{F_i^{\bf a'}} \varphi(x) \ d\mathcal{H}^{n-1}(x)\right] \\
& \ \ \ \ \ \ \ \ \ \ \leq \sup_{\phi\in C^0(\mathbb{R}^{n-1}), |\phi|\leq 1}\left[\int_{L(F_i^{\bf a})} \phi(y) \ dy - \int_{L(F_i^{\bf a'}+u_i)} \phi(y) \ dy\right].
\end{align*}
We now prove equality by showing each $\phi\in C^0(\mathbb{R}^{n-1})$ with $|\phi|\leq 1$ can be obtained this way. Given $\phi \in C^0(\mathbb{R}^{n-1})$ with $|\phi|\leq 1$, define $\varphi:\mathbb{R}^n\to \mathbb{R}$ by
$$\varphi(x) = \phi(L(x-\la x,\nu_i\ra \nu_i)) = \phi(L(x)-\la x,\nu_i \ra L(\nu_i)).$$
For any $t\in \mathbb{R}$ we have
$$\varphi(x+t\nu_i) = \phi(L(x+t\nu_i)-\la x+t\nu_i,\nu_i \ra L(\nu_i))  = \phi(L(x)-\la x,\nu_i\ra L(\nu_i)) = \varphi(x)$$
so that $\varphi\in C_{\nu_i}^0(\mathbb{R}^n)$ and $|\varphi|\leq 1$.

\noindent
\textit{Step 3:} We now show \eqref{Sym Diff Perturbed Parallel Facet Bound}. Let $\varphi \in C_{\nu_i}^1(\mathbb{R}^n)$. Then $\la \nabla \varphi(x), \nu_i \ra \equiv 0$ and since $\nu_i$ is a constant vector field, $\Div(\nu_i)\equiv 0$. Hence,
$$\int_{ K^{\bf a}} \Div(\varphi(x)\nu_i) \ dx = \int_{K^{\bf a}} \la \nabla \varphi(x), \nu_i\ra + \varphi(x)\Div(\nu_i) \ dx = 0.$$
In particular, applying the divergence theorem on $K^{\bf a}$ with the vector field $X(x) = \varphi(x) \nu_i$ yields
\begin{align*}
0 &=\int_{ K^{\bf a}} \Div(\varphi(x)\nu_i) \ dx\\
&= \int_{\partial^* K^{\bf a}} \la \varphi(x)\nu_i,\nu_{K^{\bf a}}(x) \ra \ d\mathcal{H}^{n-1}(x)=\sum_{j\in \mathscr{I}} \int_{F_j^{\bf a}} \varphi(x) \la \nu_i, \nu_j \ra \ d\mathcal{H}^{n-1}(x)\\
&= \int_{F_i^{\bf a}} \varphi(x) \ d\mathcal{H}^{n-1}(x)-\sum_{j\neq i} \cos(\theta_{ij})\int_{F_j^{\bf a}} \varphi(x) \ d\mathcal{H}^{n-1}(x)
\end{align*}
since $\nu_{K^{\bf a}}(x) \equiv \nu_j$ on $\Int(F_j^{\bf a})$ and $\cos(\theta_{ij}):=-\la \nu_i,\nu_j\ra$. Hence,
\begin{equation}\label{Facet Perturb Divergence Theorem}
\int_{F_i^{\bf a}} \varphi(x) \ d\mathcal{H}^{n-1}(x) = \sum_{j\neq i} \cos(\theta_{ij})\int_{F_j^{\bf a}} \varphi(x) \ d\mathcal{H}^{n-1}(x).
\end{equation}
By Step 2, we need only bound the difference
\begin{equation}
I := \int_{F_i^{\bf a}} \varphi(x) \ d\mathcal{H}^{n-1}(x) - \int_{F_i^{\bf a'}} \varphi(x) \ d\mathcal{H}^{n-1}(x).
\end{equation}
With the exact same logic we prove \eqref{Facet Perturb Divergence Theorem} with ${\bf a'}$ in place of ${\bf a}$. Combining this fact with \eqref{Facet Perturb Divergence Theorem} we see
\begin{align}\label{Relate perturbed Facet to others}
I &= \sum_{j\neq i} \cos(\theta_{ij})\left[\int_{F_j^{\bf a}} \varphi(x) \ d\mathcal{H}^{n-1}(x)- \int_{F_j^{\bf a'}} \varphi(x) \ d\mathcal{H}^{n-1}(x)\right] \nonumber \\
&\leq \sum_{j\neq i} \left[\int_{F_j^{\bf a}} \varphi(x) \ d\mathcal{H}^{n-1}(x)- \int_{F_j^{\bf a'}} \varphi(x) \ d\mathcal{H}^{n-1}(x)\right]
\end{align}
for all $\varphi \in C_{\nu_i}^1(\mathbb{R}^n)$. Notice that \eqref{Relate perturbed Facet to others} relies only on the $C^0$ data of $\varphi$. Approximating $\varphi \in C_{\nu_i}^0(\mathbb{R}^n)$ by $\varphi \in C_{\nu_i}^1(\mathbb{R}^n)$ shows that \eqref{Relate perturbed Facet to others} holds for $\varphi\in C_{\nu_i}^0(\mathbb{R}^n)$. In particular, taking the supremum over such $\varphi$ with $|\varphi|\leq 1$ and applying Step 2 on both sides,
\begin{align}
\mathcal{H}^{n-1}(F_i^{\bf a}\Delta(F_i^{\bf a'}+u_i)) &= \sup_{\varphi \in C_{\nu_i}^0(\mathbb{R}^n), |\varphi|\leq 1}\left[\int_{F_i^{\bf a}} \varphi(x) \ d\mathcal{H}^{n-1}(x) -\int_{F_i^{\bf a'}} \varphi(x) \ d\mathcal{H}^{n-1}(x)\right] \nonumber\\
&\leq \sum_{j\neq i} \sup_{\varphi \in C_{\nu_i}^0(\mathbb{R}^n), |\varphi|\leq 1}\left[\int_{F_j^{\bf a}} \varphi(x) \ d\mathcal{H}^{n-1}(x)- \int_{F_j^{\bf a'}} \varphi(x) \ d\mathcal{H}^{n-1}(x)\right] \nonumber \\
&\leq \sum_{j\neq i} \sup_{\varphi \in C^0(\mathbb{R}^n), |\varphi|\leq 1}\left[\int_{F_j^{\bf a}} \varphi(x) \ d\mathcal{H}^{n-1}(x)- \int_{F_j^{\bf a'}} \varphi(x) \ d\mathcal{H}^{n-1}(x)\right] \nonumber \\
&= \sum_{j\neq i} \mathcal{H}^{n-1}(F_j^{\bf a}\Delta F_j^{\bf a'}),
\end{align}
where the third line follows from the second as $C_{\nu_i}^0(\mathbb{R}^n)\subset C^0(\mathbb{R}^n)$.
This completes the proof.
\end{proof}
Having found a bound for $\mathcal{H}^{n-1}(F_i^{\bf a}\Delta (F_i^{\bf a'}+u_i))$ in terms of $\mathcal{H}^{n-1}(F_j^{\bf a}\Delta F_j^{\bf a'})$ for $j\neq i$, we now bound the latter by $\|{\bf a}-{\bf a'}\|_{\infty}$.
\begin{lem}[Bound on remaining facets]\label{Symmetric Difference Perturbed Other Facet Bound} With the setup in Lemma \ref{Symmetric Difference Perturbed Parallel Facet Bound}, for $j\neq i$
\begin{equation}\label{Sym Diff Perturbed Other Facet Bound}
\begin{cases} \mathcal{H}^{n-1}(F_j^{\bf a}\Delta F_j^{\bf a'}) = 0 \ & \ \text{if} \ j\in \mathscr{N}^{\bf a}(i)^c\cap \mathscr{N}^{\bf a'}(i)^c, \\
\mathcal{H}^{n-1}(F_j^{\bf a}\Delta F_j^{\bf a'}) \leq CM_{\Phi}^{n-1}\|{\bf a}-{\bf a'}\|_{\ell^{\infty}} \ & \ 
\text{otherwise.}\end{cases}
\end{equation}
\end{lem}
Here and later $\mathscr{N}^{\bf a}(i)$ denotes the collection of indices
\begin{equation}
\mathscr{N}^{\bf a}(i) :=\{j \in \mathscr{I} \ | \ F_i^{\bf a}\cap F_j^{\bf a} \neq \emptyset\}.
\end{equation}
I.e., those facets which neighbor $F_i^{\bf a}$ (on a possibly lower dimensional set).
\begin{proof} Suppose first that $j\in \mathscr{N}^{\bf a}(i)^c\cap \mathscr{N}^{\bf a'}(i)^c$, so that $F_i^{\bf a}\cap F_j^{\bf a} = F_i^{\bf a'}\cap F_j^{\bf a'} = \emptyset$. We show $F_j^{\bf a}=F_j^{\bf a'}$. Recall that $F_j^{\bf a} = \Co(V(F_j^{\bf a}))$, where $V(F_j^{\bf a})$ is the set of vertices of $F_j^{\bf a}$ and $\Co(\cdot)$ is the convex hull. It suffices then to show $V(F_j^{\bf a})=V(F_j^{\bf a'})$. Let $v_j^{\bf a} \in V(F_j^{\bf a})$. By definition there exists a maximal collection
\begin{equation}
\mathscr{I}_{v_j^{\bf a}} := \{l\in \mathscr{I} \ | \ \la v_j^{\bf a},\nu_l\ra = {\bf d_{\it{l}}}(1+{\bf a_{\it{l}}})\}
\end{equation}
with $|\mathscr{I}_{v_j^{\bf a}}|\geq n$; again these index the facets of $K^{\bf a}$ which contain $v_j^{\bf a}$. Since $F_i^{\bf a}\cap F_j^{\bf a} = \emptyset$, it follows that $i\not\in \mathscr{I}_{v_j^{\bf a}}$, and since ${\bf a_{\it{j}}}-{\bf a_{\it{j}}'} = a_i\delta_{ij}$ we can have in particular ${\bf a_{\it{l}}}-{\bf a_{\it{l}}'}=0$ for $l\in \mathscr{I}_{v_j^{\bf a}}$. Hence also
\begin{equation}
\mathscr{I}_{v_j^{\bf a}} = \{l\in \mathscr{I} \ | \ \la v_j^{\bf a},\nu_l\ra = {\bf d_{\it{l}}}(1+{\bf a_{\it{l}}'})\},
\end{equation}
implying $v_j^{\bf a} \in V(F_j^{\bf a'})$, as it is contained in at least $n$ facets of $K^{\bf a'}$. The reverse containment follows by an analogous argument.

Next, without loss of generality we assume ${\bf a_{\it{i}}}>{\bf a_{\it{i}}'}$ and let $j\in \mathscr{N}^{\bf a}(i)\cup \mathscr{N}^{\bf a'}(i)$, $j\neq i$. Recall $x\in F_j^{\bf a'}$ if and only if $\la x,\nu_j\ra = {\bf d_{\it{j}}}(1+{\bf a_{\it{j}}'})$ and $\la x,\nu_l \ra \leq {\bf d_{\it{l}}}(1+{\bf a_{\it{l}}'})$ for all $l\in \mathscr{I}$. Additionally
\begin{equation}
{\bf a_{\it{l}}'} = {\bf a_{\it{l}}} \ \ \text{for} \ l\neq i \ \ \ \ \text{and} \ \ \ \ {\bf a_{\it{i}}}> {\bf a_{\it{i}}'}.
\end{equation}
In particular, since $j\neq i$, $x\in F_j^{\bf a}$ too and we have that $F_j^{\bf a'}\subset F_j^{\bf a}$. So, $\mathcal{H}^{n-1}(F_j^{\bf a}\Delta F_j^{\bf a'}) = \mathcal{H}^{n-1}(F_j^{\bf a}\setminus F_j^{\bf a'})$. 

The idea now is to construct a nice set $E$ such that $F_j^{\bf a}\setminus F_j^{\bf a'}\subset E$ and $\mathcal{H}^{n-1}(E)\lesssim \|{\bf a}-{\bf a'}\|_{\ell^{\infty}}$. Consider the following sets:
\begin{align*}
S &= \left\lbrace \bigg|\la x,\nu_i \ra - {\bf d_{\it{i}}}\left(1+\frac{{\bf a_{\it{i}}}+{\bf a_{\it{i}}'}}{2}\right)\bigg| \leq \frac{{\bf d_{\it{i}}}}{2}|{\bf a_{\it{i}}}-{\bf a_{\it{i}}'}| \right\rbrace,\\
Q_R &= \bigcap_{l=1}^n \{|\la x,u_l \ra| \leq R\}, \ \text{and}\\
\Sigma_{\nu_j}^{1+{\bf a_{\it{j}}}} &=\{\la x,\nu_j \ra = {\bf d_{\it{j}}}(1+{\bf a_{\it{j}}})\}
\end{align*}
where $\{u_1,...,u_{n-2}\}$ is an orthonormal basis for $\Span(\{\nu_i,\nu_j\})^{\perp}$, 
$$u_{n-1} = \frac{\nu_i - \la \nu_i,\nu_j \ra \nu_j}{|\nu_i - \la \nu_i,\nu_j \ra \nu_j|} = \csc(\theta_{ij})\nu_i+\cot(\theta_{ij})\nu_j,$$
the co-normal to $F_j$ in the direction of $\nu_i$, and $u_n=\nu_j$. We let $E = S\cap Q_R\cap \Sigma_{\nu_j}^{1+{\bf a_{\it{j}}}}$ with $R$ to be decided momentarily. The purpose of the slab $S$ is to restrict the height of $E$, of $Q_R$ to restrict the $(n-2)$-dimensional ``width", and of $\Sigma_{\nu_j}^{1+{\bf a_{\it{j}}}}$ to ensure that $E$ is $(n-1)$-dimensional and $F_j^{\bf a}\setminus F_j^{\bf a'}\subset E$. Indeed, it follows readily from the definition of a facet, together with the fact that we consider only a perturbation in the $\nu_i$ direction, that
$$F_j^{\bf a}\setminus F_j^{\bf a'}\subset S\cap \Sigma_{\nu_j}^{1+{\bf a_{\it{j}}}}.$$
Of course, this set is unbounded and hence not a good candidate for $E$. On the other hand, $K^{\bf a}\subset B_{M_{\Phi}^{\bf a}}(0) \subset B_R(0)\subset Q_R$, which is bounded, for any $R\geq M_{\Phi}^{\bf a}$. Thus we let $E = S\cap Q_{M_{\Phi}^{\bf a}}\cap \Sigma_{\nu_j}^{1+{\bf a_{\it{j}}}}$, which evidently satisfies the desired properties. Finally, the specific choice of $Q_{M_{\Phi}^{\bf a}}$ allows us to easily bound $\mathcal{H}^{n-1}(E)$. It is oriented according to $\{u_1,...,u_n\}$ so that, in essence, $E$ is a box with $(n-2)$-dimensional width $2M_{\Phi}^{\bf a}$ a (slant) height $\csc(\theta_{ij})\cdot{\bf d_{\it{i}}}({\bf a_{\it{i}}}-{\bf a_{\it{i}}'})$. It follows that $E$ is isometric to a box in $\mathbb{R}^{n-1}$ having $n-2$ sides with length $2M_{\Phi}^{\bf a}$ and one side of length $\csc(\theta_{ij})\cdot {\bf d_{\it{i}}}({\bf a_{\it{i}}}-{\bf a_{\it{i}}'})$. The formal verification of this claim is deferred to the end of this proof. Figure \ref{Perturb Bound Figure} elucidates the situation.
\begin{figure}
\hspace{-20pt}
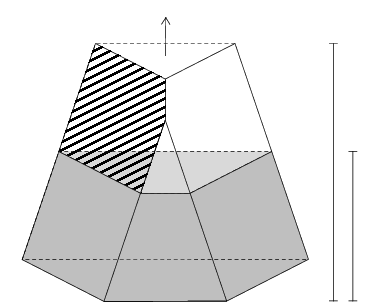
\hspace{20pt}
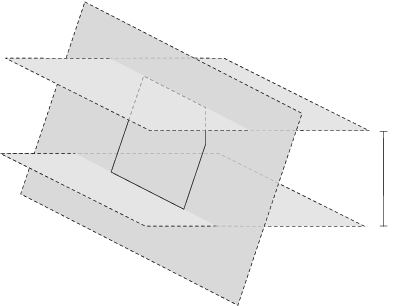
\caption{Left: A parallel perturbation in the direction $\nu_i$, with $F_j^{\bf a}\Delta F_j^{\bf a'}$ emphasized. Right: The set $F_j^{\bf a}\Delta F_j^{\bf a'}$ is contained within a slab $S$ of height $|{\bf a_{\it{i}}}-{\bf a_{\it{i}}'}|$ and the plane $\Sigma_{\nu_j}^{1+{\bf a_{\it{j}}}}$. As $F_j^{\bf a}\Delta F_j^{\bf a'}$ is slanted (according to $\nu_j$), its height is bounded by $\csc(\theta_{ij})\cdot {\bf d_{\it{i}}}|{\bf a_{\it{i}}}-{\bf a_{\it{i}}'}|$.}\label{Perturb Bound Figure}
\end{figure}
Thus,
\begin{equation}
\mathcal{H}^{n-1}(F_j^{\bf a}\Delta F_j^{\bf a'}) \leq \mathcal{H}^{n-1}(E) = {\bf d_{\it{i}}}(2M_{\Phi}^{\bf a})^{n-2}\csc(\theta_{ij})|{\bf a_{\it{i}}}-{\bf a_{\it{i}}'}|.
\end{equation}
Recall \cite[Theorem 1.8.11]{Schneider} that \eqref{Hausdorff ell infinity} holds and so
\begin{equation}
d_H(K,K^{\bf a})=\|f^{\bf a}-f\|_{L^{\infty}(S^{n-1})}.
\end{equation}
Thus for any $\nu\in S^{n-1}$,
\begin{equation}
f^{\bf a}(\nu) \leq |f^{\bf a}(\nu)-f(\nu)|+f(\nu) \leq d_H(K,K^{\bf a})+M_{\Phi}.
\end{equation}
Taking the supremum over $\nu\in S^{n-1}$ and applying Lemma \ref{Equivalence}, there exists $C(n,K)>0$ such that
\begin{equation}
M_{\Phi}^{\bf a} \leq M_{\Phi}+d_H(K,K^{\bf a}) \leq CM_{\Phi}
\end{equation}
as $\|{\bf a}\|_{\ell^{\infty}}<1$. Of course, we can absorb $M_{\Phi}$ into this constant but we separate them to highlight the scaling. In total,
\begin{align}
\mathcal{H}^{n-1}(F_j^{\bf a}\Delta F_j^{\bf a'}) &\leq {\bf d_{\it{i}}}(2CM_{\Phi})^{n-2}\csc(\theta_{ij})|{\bf a_{\it{i}}}-{\bf a_{\it{i}}'}|\\
&\leq CM_{\Phi}^{n-1}\csc(\theta_{ij})\|{\bf a}-{\bf a'}\|_{\ell^{\infty}}
\end{align}
since ${\bf d_{\it{i}}}=f(\nu_i)\leq \sup_{\nu\in S^{n-1}}f(\nu) = M_{\Phi}$.

To prove that $\mathcal{H}^{n-1}(E) = {\bf d_{\it{i}}}(2M_{\Phi}^{\bf a})^{n-2}\csc(\theta_{ij})|{\bf a_{\it{i}}}-{\bf a_{\it{i}}'}|$ we show $\partial E$ contains orthogonal segments of the desired lengths, and that the segments cannot be extended. Set
$$v = v_1\nu_i+v_2\nu_j, \ \ \ \ \ \begin{pmatrix}
v_1 \\
v_2
\end{pmatrix} = \csc(\theta_{ij})\begin{pmatrix}
\csc(\theta_{ij}) & \cot(\theta_{ij}) \\
\cot(\theta_{ij}) & \csc(\theta_{ij}) \\
\end{pmatrix}\begin{pmatrix}
{\bf d_{\it{i}}}(1+{\bf a_{\it{i}}})\\
{\bf d_{\it{j}}}(1+{\bf a_{\it{j}}})
\end{pmatrix}.$$
The vector $v$ is chosen precisely so that $\la v,u_1\ra =... = \la v,u_{n-2}\ra = 0$ and
\begin{align*}
\la v,\nu_i \ra &= v_1-\cos(\theta_{ij})v_2 = {\bf d_{\it{i}}}(1+{\bf a_{\it{i}}})\\
\la v,\nu_j \ra &= -\cos(\theta_{ij})v_1+v_2 = {\bf d_{\it{j}}}(1+{\bf a_{\it{j}}}).
\end{align*}
Together these imply that $tu_l +v \in \partial S \cap Q_{M_{\Phi}^{\bf a}}\cap \Sigma_{\nu_j}^{1+{\bf a_{\it{j}}}}\subset \partial E$ for $t$ sufficiently small. With this choice, the maximal domain for $l=1,...,n-2$ is $[-M_{\Phi}^{\bf a},M_{\Phi}^{\bf a}]$. Now consider for $t\in [t^-,t^+]$ the segment $\gamma_{n-1}(t)= -tu_{n-1}+v$ with $v$ as before. We wish to find the maximal interval $[t^-,t^+]$ such that $\gamma_{n-1}([t^-,t^+])\subset E$. That is, we find $t^-, t^+$ such that
$$\gamma_{n-1}(t^-) = {\bf d_{\it{i}}}(1+{\bf a_{\it{i}}}), \ \ \ \ \ \ \gamma_{n-1}(t^+) = {\bf d_{\it{i}}}(1+{\bf a_{\it{i}}'}).$$
Notice that at $t=0$ we have 
$$\gamma_{n-1}(0)=\la v,\nu_i \ra = {\bf d_{\it{i}}}(1+{\bf a_{\it{i}}}).$$ 
On the other hand, by construction of $u_{n-1}$,
\begin{align}
\la \gamma_{n-1}(t), \nu_i \ra &= \la -tu_{n-1}+v, \nu_i \ra \\
&= -t\la \csc(\theta_{ij})\nu_i+\cot(\theta_{ij})\nu_j, \nu_i \ra+{\bf d_{\it{i}}}(1+{\bf a_{\it{i}}})= -t\sin(\theta_{ij})+{\bf d_{\it{i}}}(1+{\bf a_{\it{i}}})
\end{align}
so that at $t^+ = \csc(\theta_{ij})\cdot {\bf d_{\it{i}}}({\bf a_{\it{i}}}-{\bf a_{\it{i}}'})>0$ we have 
$$\gamma_{n-1}(t^+) = {\bf d_{\it{i}}}(1+{\bf a_{\it{i}}'}).$$
So, the maximal interval is $[0,\csc(\theta_{ij})\cdot {\bf d_{\it{i}}}({\bf a_{\it{i}}}-{\bf a_{\it{i}}'})]$. It follows that $E$ is isometric to a box in $\mathbb{R}^{n-1}$ having $n-2$ sides with length $2M_{\Phi}^{\bf a}$ and one side of length $\csc(\theta_{ij})\cdot {\bf d_{\it{i}}}({\bf a_{\it{i}}}-{\bf a_{\it{i}}'})$, as desired.
\end{proof}
We are now ready to generally bound the $L^1$ norm of parallel facets under a parallel perturbation.
\begin{proof}[Proof of Proposition 3.1] First notice that if ${\bf a}$ and ${\bf a'}$ are a parallel perturbation in the direction $\nu_i$ then for some $C(n,K)>0$:
\begin{align}
\mathcal{H}^{n-1}(F_j^{\bf a}\Delta F_j^{\bf a'}) &=0, \ \ \ \ \ \ \ \ \ \ \ \ \ \ \ \ \ \ \ \ \ \ \ \ \ \ \text{if} \ j\in \mathscr{N}^{\bf a}(i)^c\cap \mathscr{N}^{\bf a'}(i)^c\ \text{and} \label{estimate null}\\
\mathcal{H}^{n-1}(F_j^{\bf a}\Delta F_j^{\bf a'}) &\leq CM_{\Phi}^{n-1}\|{\bf a}-{\bf a'}\|_{\ell^{\infty}}, \ \ \ \text{if} \ j\in \mathscr{N}^{\bf a}(i)\cup \mathscr{N}^{\bf a'}(i), \ j\neq i \label{estimate affected}
\end{align}
by Lemma \ref{Symmetric Difference Perturbed Other Facet Bound}. So, we are only missing an $\ell^{\infty}$ bound on $\mathcal{H}^{n-1}(F_i^{\bf a}\Delta (F_i^{\bf a'}+u_i))$. However, by Lemma \ref{Symmetric Difference Perturbed Parallel Facet Bound}, we can estimate this in terms of the contributions from $j\neq i$, so that
\begin{equation}
\mathcal{H}^{n-1}(F_i^{\bf a}\Delta (F_i^{\bf a'}+u_i)) \leq \sum_{j\neq i}\mathcal{H}^{n-1}(F_j^{\bf a}\Delta F_j^{\bf a'}) \leq  CM_{\Phi}^{n-1}\|{\bf a}-{\bf a'}\|_{\ell^{\infty}}.\label{estimate perturb}
\end{equation}
Now given generic $K^{\bf a}, K^{\bf a'}$ we set ${\bf a_0} = {\bf a'}$ and define $\bf a_i$ for $i=1,...,N$ by
$${\bf a_{i,\it{j}}}-{\bf a_{i-1,\it{j}}} = ({\bf a_{\it{i}}}-{\bf a_{\it{i}}'})\delta_{ij}, \ \ \ j=1,...,N.$$
Note that ${\bf a_i}$ and ${\bf a_{i-1}}$ are a parallel perturbation in the direction $\nu_i$. So we can apply the estimates \eqref{estimate null}, \eqref{estimate affected}, and \eqref{estimate perturb} as above for each pair $K^{\bf a_{i-1}}$ and $K^{\bf a_i}$. Defining $u_i = {\bf d_{\it{i}}}({\bf a_{\it{i}}}-{\bf a_{\it{i}}'})\nu_i$,
\begin{align}
\mathcal{H}^{n-1}(F_i^{\bf a}\Delta (F_i^{\bf a'}+u_i)) & \leq \sum_{j=1}^{i-1}\mathcal{H}^{n-1}((F_i^{\bf a_j}+u_i)\Delta (F_i^{\bf a_{j-1}}+u_i))+\sum_{j=i+1}^{N}\mathcal{H}^{n-1}(F_i^{\bf a_j}\Delta F_i^{\bf a_{j-1}}) \\
& \ \ \ \ + \mathcal{H}^{n-1}(F_i^{\bf a_i}\Delta (F_i^{\bf a_{i-1}}+u_i)) \\
&\leq \mathcal{H}^{n-1}(F_i^{\bf a_i}\Delta (F_i^{\bf a_{i-1}}+u_i))+\sum_{j\neq i}\mathcal{H}^{n-1}(F_i^{\bf a_j}\Delta F_i^{\bf a_{j-1}}). \label{initial facet bound parallel deform}
\end{align}
By \eqref{estimate perturb}, the first term in \eqref{initial facet bound parallel deform} is estimated as
\begin{equation}\label{facet bound parallel deform 1}
\mathcal{H}^{n-1}(F_i^{\bf a_i}\Delta (F_i^{\bf a_{i-1}}+u_i)) \leq  CM_{\Phi}^{n-1}\|{\bf a}-{\bf a'}\|_{\ell^{\infty}}
\end{equation}
while the latter term in \eqref{initial facet bound parallel deform} is again estimated using \eqref{estimate null} and \eqref{estimate affected} as
\begin{equation}\label{facet bound parallel deform 2}
\sum_{j\neq i}\mathcal{H}^{n-1}(F_i^{\bf a_j}\Delta F_i^{\bf a_{j-1}}) \leq CM_{\Phi}^{n-1}\|{\bf a}-{\bf a'}\|_{\ell^{\infty}}.
\end{equation}
Together, \eqref{facet bound parallel deform 1} and \eqref{facet bound parallel deform 2} show the claim.
\end{proof}

\subsection{Strong quantitative stability for parallel polytopes}
Our goal in this section is to prove a version of Theorem \ref{Main Result} applicable to polytopes $K^{\bf a}$ obtained from Theorem \ref{Replacement}. This is achieved if we can prove the result among the class $\mathcal{C}(K)\subset \mathcal{C}_{\text{par}}(K)$ of parallel polytopes to $K$ with the same volume as $K$. Precisely, we prove:
\begin{prop}\label{Result for Parallel Polytopes} There exists an $a_0(n,K)>0$ and a constant $C(n,K)>0$ such that whenever $K^{\bf a}\in \mathcal{C}(K)$ with $\|{\bf a}\|_{\ell^{\infty}}\leq a_0$, and $|K \Delta K^{\bf a}| = \inf\{|K\Delta (K^{\bf a}+x)| \ | \ x\in\mathbb{R}^n\}$, it holds
$$\beta_{\Phi}(K^{\bf a})^2 \leq C\delta_{\Phi}(K^{\bf a}).$$
\end{prop}
Before proving Proposition \ref{Result for Parallel Polytopes} let us introduce some notation and technical lemmas. Set
$${\bf m} = (\mathcal{H}^{n-1}(F_1),...,\mathcal{H}^{n-1}(F_N)), \ \ {\bf m^a} = (\mathcal{H}^{n-1}(F_1^{\bf a}),...,\mathcal{H}^{n-1}(F_N^{\bf a})).$$
In terms of these quantities, the following initial estimate on $\delta_{\Phi}(K^{\bf a})$ is easy to show.
\begin{lem}\label{Delta Estimates} Let $K^{\bf a}\in \mathcal{C}(K)$. Then,
\begin{equation}\label{delta estimate O(a)}
\delta_{\Phi}(K^{\bf a}) = -\frac{1}{n|K|}\sum_{i=1}^N {\bf d_{\it{i}}}{\bf a_{\it{i}}}{\bf m_{\it{i}}^a}.
\end{equation}
Furthermore, there exists a $C(n,K)>0$ such that if $|K \Delta K^{\bf a}| = \inf\{|K\Delta (K^{\bf a}+x)| \ | \ x\in\mathbb{R}^n\}$ then
\begin{equation}\label{delta estimate O(a^2)}
\delta_{\Phi}(K^{\bf a}) \geq C(n,K)\|{\bf a}\|_{\ell^{\infty}}^2.
\end{equation}
We write $\delta_{\Phi}(K^{\bf a}) \geq O(\|{\bf a}\|_{\ell^{\infty}}^2)$ to denote this property.
\end{lem}
The latter, \eqref{delta estimate O(a^2)}, says precisely that $\delta_{\Phi}(K^{\bf a})$ is non-degenerate with respect to $\|{\bf a}\|_{\ell^{\infty}}$.
\begin{proof} First we can express the anisotropic perimeter and volume of $K$ and $K^{\bf a}$ in terms of the above quantities as 
$$\Phi(K) = \sum_{i=1}^{N}{\bf d_{\it{i}}}{\bf m_{\it{i}}}, \ \ \ \Phi(K^{\bf a}) = \sum_{i=1}^{N}{\bf d_{\it{i}}}{\bf m_{\it{i}}^a}, \ \ \ |K| = \frac{1}{n}\sum_{i=1}^{N}{\bf d_{\it{i}}}{\bf m_{\it{i}}}, \ \ \ |K^{\bf a}| = \frac{1}{n}\sum_{i=1}^{N}{\bf d_{\it{i}}}(1+{\bf a_{\it{i}}}){\bf m_{\it{i}}^a}.$$
Since $K^{\bf a}\in \mathcal{C}(K)$ we have by definition $|K^{\bf a}|=|K|$. Applying this volume constraint yields
\begin{equation}\label{volume constraint}
\frac{1}{n}\sum_{i=1}^N {\bf d_{\it{i}}}{\bf m_{\it{i}}} = |K|=|K^{\bf a}|=\frac{1}{n}\sum_{i=1}^N{\bf d_{\it{i}}}{\bf m_{\it{i}}^a} + \frac{1}{n}\sum_{i=1}^N {\bf d_{\it{i}}}{\bf a_{\it{i}}}{\bf m_{\it{i}}^a}.
\end{equation}
Next, by \eqref{volume constraint} and the definition of $\delta_{\Phi}(K^{\bf a})$,
$$n|K|\delta_{\Phi}(K^{\bf a}) = \Phi(K^{\bf a})- \Phi(K) = \sum_{i=1}^N {\bf d_{\it{i}}}({\bf m_{\it{i}}^a}-{\bf m_{\it{i}}})=-\sum_{i=1}^N {\bf d_{\it{i}}}{\bf a_{\it{i}}}{\bf m_{\it{i}}^a},$$
which is exactly \eqref{delta estimate O(a)}.

As for \eqref{delta estimate O(a^2)}, by the quantitative stability with the Fraenkel asymmetry, \eqref{sharp stability}, we have that there exists $C_1(n)>0$ such that
\begin{equation}\label{initial lower bound deficit}
\frac{|K\Delta K^{\bf a}|^2}{|K|^2} = \frac{1}{|K|^2}\inf\{|K\Delta (K^{\bf a}+x)| \ | \ x\in \mathbb{R}^n\}^2 = \alpha_{\Phi}(K^{\bf a})^2 \leq C_1\delta_{\Phi}(K^{\bf a}),
\end{equation}
where we have used the fact that $K^{\bf a}\in \mathcal{C}(K)$ satisfies $|K^{\bf a}|=|K|$. Next by Lemma \ref{Equivalence} there exists a constant $C_2(n,K)>0$ such that $\|{\bf a}\|_{\ell^{\infty}}\leq C_2|K\Delta K^{\bf a}|$. Combined with \eqref{initial lower bound deficit} we have $C\|{\bf a}\|_{\ell^{\infty}}^2 \leq \delta_{\Phi}(K^{\bf a})$, with $C=1/(C_1C_2^2|K|^2)$, as desired. 
\end{proof}
Finally we provide an initial estimate on $\beta_{\Phi}(K^{\bf a})$.
\begin{lem}\label{Upper Bound on Beta} Let $K^{\bf a}\in \mathcal{C}(K)$. Then
\begin{equation}
\beta_{\Phi}(K^{\bf a})^2 \leq \delta_{\Phi}(K^{\bf a}) + \frac{n-1}{n|K|}\left(A-B\right)
\end{equation}
where
\begin{equation}\label{define A and B}
A = \sum_{i\in \mathscr{I}}\int_0^{1} \frac{{\bf d_{\it{i}}}}{r}\mathcal{H}^{n-1}(rF_i \cap (K\setminus K^{\bf a})) \ dr, \ \ \ \ \ B = \sum_{i\in \mathscr{I}}\int_1^{\infty} \frac{{\bf d_{\it{i}}}}{r}\mathcal{H}^{n-1}(rF_i \cap (K^{\bf a}\setminus K)) \ dr.
\end{equation}
\end{lem}
\begin{proof}
Recall from \eqref{gamma} that we can write $\beta_{\Phi}$ in terms of $\gamma_{\Phi}$,
$$\beta_{\Phi}(K^{\bf a})^2 = \frac{\Phi(K^{\bf a})-(n-1)\gamma_{\Phi}(K^{\bf a})}{n|K|^{1/n}|K^{\bf a}|^{(n-1)/n}} = \frac{\Phi(K^{\bf a})}{n|K|} - \frac{(n-1)\gamma_{\Phi}(K^{\bf a})}{n|K|},$$
where $\gamma_{\Phi}(E)$ is defined by
$$\gamma_{\Phi}(E) = \sup_{y\in \mathbb{R}^n}\int_E \frac{1}{f_*(x-y)} \ dx.$$
Testing at the origin we get
\begin{align}
\beta_{\Phi}(K^{\bf a})^2 &\leq  \frac{\Phi(K^{\bf a})}{n|K|}- \frac{n-1}{n|K|}\int_{K^{\bf a}} \frac{1}{f_*(x)} \ dx \\
&= \left(\frac{\Phi(K^{\bf a})-\Phi(K)}{n|K|}\right) + \frac{\Phi(K)}{n|K|}- \frac{n-1}{n|K|}\int_{K^{\bf a}} \frac{1}{f_*(x)} \ dx.
\end{align}
Note that $\Phi(K) = (n-1)\gamma_{\Phi}(K)$. Hence,
\begin{align*}
\beta_{\Phi}(K^{\bf a})^2 &\leq \delta_{\Phi}(K^{\bf a}) + \frac{n-1}{n|K|}\left(\int_{K} \frac{1}{f_*(x)} \ dx-\int_{K^{\bf a}} \frac{1}{f_*(x)} \ dx\right)\\
&= \delta_{\Phi}(K^{\bf a}) + \frac{n-1}{n|K|}\left(\int_{K\setminus K^{\bf a}} \frac{1}{f_*(x)} \ dx-\int_{K^{\bf a}\setminus K} \frac{1}{f_*(x)}\ dx\right).
\end{align*}
The weighted anisotropic co-area formula lets us compute these integrals more directly. Indeed, recall that for any Borel $g:\mathbb{R} \to [0,\infty)$ and Lipschitz $u:\mathbb{R}^n\to \mathbb{R}$, and open $\Omega\subset \mathbb{R}^n$ it states that
\begin{equation}\label{anisotropic co-area}
\int_{\Omega} f(\nabla u(x))g(f_*(x)) \ dx = \int_0^{\infty} \Phi(\{u<r\};\Omega)g(r) \ dr,
\end{equation}
where the relative anisotropic perimeter $\Phi(E;\Omega)$ is defined as 
\begin{equation}
\Phi(E;\Omega) = \int_{\partial^* E\cap \Omega} f(\nu_{E}(x)) \ d\mathcal{H}^{n-1}(x).
\end{equation}
In particular, selecting $g(r)=1/r$ and $u(x)=f_*(x)$ we obtain 
\begin{equation}\label{anisotropic co-area applied}
\int_{\Omega} \frac{1}{f_*(x)} \ dx = \int_0^{\infty} \frac{1}{r}\Phi(rK;\Omega) \ dr
\end{equation}
as $K = \{f_*<1\}$ and the Fenchel inequality,
\begin{equation}
\la x,\nu \ra \leq f_*(x)f(\nu)
\end{equation}
implies that $f(\nabla f_*(x))\equiv 1$ for a.e. $x\in \mathbb{R}^n$. For more details see \cite{Neumayer2016}. Applying \eqref{anisotropic co-area applied} with $\Omega=K\setminus K^{\bf a}$ and $\Omega=K^{\bf a}\setminus K$ gives
\begin{align*}
\int_{K\setminus K^{\bf a}} \frac{1}{f_*(x)} \ dx &= \int_0^{\infty} \frac{1}{r}\Phi(rK; K\setminus K^{\bf a}) \ dr = \sum_{i\in \mathscr{I}} \int_0^{1} \frac{{\bf d_{\it{i}}}}{r}\mathcal{H}^{n-1}(rF_i \cap (K\setminus K^{\bf a})) \ dr \\
\int_{K^{\bf a}\setminus K} \frac{1}{f_*(x)}\ dx &= \int_0^{\infty} \frac{1}{r}\Phi(rK; K^{\bf a}\setminus K) \ dr = \sum_{i\in \mathscr{I}}\int_1^{\infty} \frac{{\bf d_{\it{i}}}}{r}\mathcal{H}^{n-1}(rF_i \cap (K^{\bf a}\setminus K)) \ dr.
\end{align*}
In the latter integrals we make use of the fact that $\partial^* K = \Int(F_1)\cup...\cup \Int(F_N)$ and that $\nu_K \equiv \nu_i$ on $\Int(F_i)$, so $f(\nu_K)\equiv {\bf d_{\it{i}}}$ on $\Int(F_i)$. 
\end{proof}
In light of \eqref{define A and B}, to find the appropriate bounds on $A$ and $B$ we effectively only need to study
$$\int_0^1 \frac{1}{r}\mathcal{H}^{n-1}(rF_i \cap (K\setminus K^{\bf a})) \ dr \ \ \ \text{and} \ \ \ \int_1^{\infty} \frac{1}{r}\mathcal{H}^{n-1}(rF_i \cap (K^{\bf a}\setminus K)) \ dr$$
for $i \in \mathscr{I}$. The idea now is that whether or not these integrals contribute to $A$ and $B$ for a particular $i\in \mathscr{I}$ depends on its inclusion in one of the following collections:
\begin{align}
\mathscr{I}^- = \{i\in \mathscr{I} \ | \ {\bf a_{\it{i}}}<0\} \ \ \ \text{and} \ \ \ \mathscr{I}^+ = \mathscr{I}\setminus \mathscr{I}^+. \label{index collections}
\end{align}
To control the first integral we will subdivide these collections so as to estimate $\mathcal{H}^{n-1}(rF_i\cap (K\setminus K^{\bf a}))$ easily. As for the latter integral, we will show that $\mathcal{H}^{n-1}(rF_i\cap (K\setminus K^{\bf a}))$ can be written as $\mathcal{H}^{n-1}(F_i^{\bf a_1}\cap F_i^{\bf a_2})$ for judicious choices of ${\bf a_1}$ and ${\bf a_2}$, then make use of Proposition \ref{Symmetric Difference Perturbed Facet Bound} and Corollary \ref{Asymptotics of Parallel Facet Areas}.

Let us observe the following simple facts. By convexity $rF_i \subset K$ if and only if $0\leq r<1$. On the other hand, 
$K^{\bf a}\subset \{\la x,\nu_i \ra <{\bf d_{\it{i}}}(1+{\bf a_{\it{i}}})\}$ while 
$rF_i \subset \{\la x,\nu_i \ra = r{\bf d_{\it{i}}}\}$. It follows that $rF_i \not\subset K^{\bf a}$ whenever $1+{\bf a_{\it{i}}}\leq r$. 

Focusing on the case $0<r<1$, observe that if $i\in \mathscr{I}^-$ and $1+{\bf a_{\it{i}}}\leq r$ then $rF_i \subset K\setminus K^{\bf a}$, and so we have 
\begin{equation}
\mathcal{H}^{n-1}(rF_i \cap (K\setminus K^{\bf a})) = \mathcal{H}^{n-1}(rF_i) = r^{n-1}{\bf m_{\it{i}}}.
\end{equation}
When $r<1+{\bf a_{\it{i}}}$, for any $i\in \mathscr{I}$, we must be more careful. While it is possible that $rF_i \subset K\setminus K^{\bf a}$ in part of this regime, we will ignore this. Instead we bound $\mathcal{H}^{n-1}(rF_i \cap (K\setminus K^{\bf a}))$ in terms of contributions from other sides, similar to Lemma \ref{Symmetric Difference Perturbed Other Facet Bound}. To start we define
\begin{align}\label{slab}
S_{\nu_j}^{1+{\bf a_{\it{j}}},r}&= \{{\bf d_{\it{j}}}(1+{\bf a_{\it{j}}}) \leq \la x,\nu_j \ra \leq r{\bf d_{\it{j}}} \}\\
&=\left\lbrace \bigg|\la x,\nu_j \ra - \frac{\bf d_{\it{j}}}{2}(1+{\bf a_{\it{j}}}+r)\bigg| \leq \frac{\bf d_{\it{j}}}{2}|{1+\bf a_{\it{j}}}-r| \right\rbrace
\end{align}
as the slab between the parallel hyperplanes $\Sigma_{\nu_j}^{1+{\bf a_{\it{j}}}}$ and $\Sigma_{\nu_j}^r$ whenever $1+{\bf a_{\it{j}}}<r$. Observe now that for any $i\in \mathscr{I}$  and $0<r<1+{\bf a_{\it{i}}}$ if $rF_i \cap (K\setminus K^{\bf a})\neq \emptyset$ then there exists a (minimal) collection of indices $\mathscr{I}_i^{\bf a}(r)\subset \mathscr{I}^-$ such that
\begin{equation}\label{slab decomposition}
rF_i \cap (K\setminus K^{\bf a})\subset \bigcup_{j\in \mathscr{I}_i^{\bf a}(r)} S_{\nu_j}^{1+{\bf a_{\it{j}}},r}.
\end{equation}
See Figure \ref{Scaling} for more information. Indeed, taking $x\in rF_i \cap (K\setminus K^{\bf a})$ there exists $j\in \mathscr{I}$ such that $\la x,\nu_j \ra \geq {\bf d_{\it{j}}}(1+{\bf a_{\it{j}}})$. On the other hand, since $x\in rF_i$ we have $\la x,\nu_{j'}\ra \leq r{\bf d_{\it{j'}}}$ for any $j'\in \mathscr{I}$. It follows that $x\in S_{\nu_j}^{1+{\bf a_{\it{j}}},r}$. 
\begin{figure}
\hspace{-20pt}
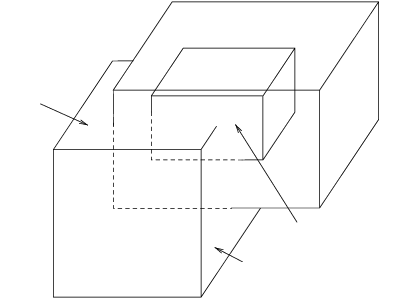
\hspace{20pt}
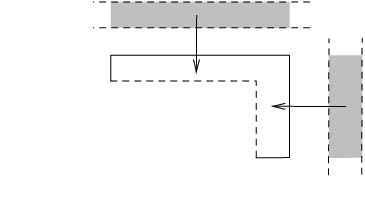
\caption{Left: $K, K^{\bf a},$ and a particular $rK$ with $0<r<1$. Right: We can bound $rF_i\cap (K\setminus K^{\bf a})$ in terms of the contributions between the parallel hyperplanes $S_{\nu_j}^{1+{\bf a_{\it{j}}},r}$. In this case, $\mathscr{I}_i^{\bf a}(r)=\{j_1,j_2\}$.}\label{Scaling}
\end{figure}

The following lemma shows this decomposition provides a good estimate.
\begin{lem}\label{Transverse Estimate} Let $i\in \mathscr{I}$ and $0<r<\min\{1,1+{\bf a_{\it{i}}}\}$ be such that  $rF_i\cap (K\setminus K^{\bf a})\neq \emptyset$. Then there exists $C(n,K)>0$ such that
\begin{equation}
\mathcal{H}^{n-1}(rF_i \cap (K\setminus K^{\bf a})) \leq C\sum_{j \in \mathscr{I}_i^{\bf a}(r)} (r-(1+{\bf a_{\it{j}}})).
\end{equation}
\end{lem}
\begin{proof} By choice of $\mathscr{I}_i^{\bf a}(r)$ we have \eqref{slab decomposition}, implying
\begin{align}
rF_i \cap (K\setminus K^{\bf a})\subset \left(\bigcup_{j\in \mathscr{I}_i^{\bf a}(r)}S_{\nu_j}^{1+{\bf a_{\it{j}}},r}\right)\cap B_{M_{\Phi}}(0)\cap \Sigma_{\nu_i}^r.
\end{align}
Using the exact same logic as in Lemma \ref{Symmetric Difference Perturbed Other Facet Bound} we estimate
\begin{equation}
\mathcal{H}^{n-1}(rF_i \cap (K\setminus K^{\bf a})) \leq \sum_{j \in \mathscr{I}_i^{\bf a}(r)} {\bf d_{\it{j}}}(2M_{\Phi})^{n-2}\csc(\theta_{ij})(r-(1+{\bf a_{\it{j}}})),
\end{equation}
which is non-negative owing to the fact that $r>1+{\bf a_{\it{j}}}$ whenever $j\in \mathscr{I}_i^{\bf a}(r)$. 
\end{proof}
Before moving to the proof, let us summarize the argument. The goal will be to show that there exists $C(n,K)>0$ such that
\begin{equation}
A \leq -C\sum_{i\in \mathscr{I}^-}{\bf d_{\it{i}}}{\bf a_{\it{i}}}{\bf m_{\it{i}}^a}+O(\|{\bf a}\|_{\ell^{\infty}}^2) \ \ \ \ \text{and} \ \ \  -B \leq -C\sum_{i\in \mathscr{I}^+} {\bf d_{\it{i}}}{\bf a_{\it{i}}}{\bf m_{\it{i}}^a}+O(\|{\bf a}\|_{\ell^{\infty}}^2).
\end{equation}
Indeed, with this we have
\begin{equation}
A-B\leq C\left(-\sum_{i\in \mathscr{I}}{\bf d_{\it{i}}}{\bf a_{\it{i}}}{\bf m_{\it{i}}^a}\right)+O(\|{\bf a}\|_{\ell^{\infty}}^2) \leq C\delta_{\Phi}(K^{\bf a})
\end{equation}
after applying the estimates \eqref{delta estimate O(a)} and \eqref{delta estimate O(a^2)} on $\delta_{\Phi}(K^{\bf a})$. The issue is that some of the terms are $O({\bf a_{\it{i}}})$. In order to apply \eqref{delta estimate O(a)}, they must be paired with ${\bf d_{\it{i}}}{\bf m_{\it{i}}^a}$, and such a term must be represented for each $i\in \mathscr{I}$. Due to the sign, this is only relevant for $i\in \mathscr{I}^+$.   

\begin{proof}[Proof of Proposition \ref{Result for Parallel Polytopes}] $\,$

\noindent
\textit{Step 1: Asymptotics of $A$}. Rewriting $A$, we have
\begin{align}
A &= \sum_{i\in \mathscr{I}^-}\int_0^{1} \frac{{\bf d_{\it{i}}}}{r}\mathcal{H}^{n-1}(rF_i \cap (K\setminus K^{\bf a})) \ dr+\sum_{i\in \mathscr{I}^+}\int_0^{1} \frac{{\bf d_{\it{i}}}}{r}\mathcal{H}^{n-1}(rF_i \cap (K\setminus K^{\bf a})) \ dr. \label{decompose A initial}
\end{align}
Setting
\begin{equation}
r_i = \min\{1,\max\{r>0 \ | \ rF_i \subset K^{\bf a}\}\},
\end{equation}
we continue and write
\begin{align}
A &= \sum_{i\in \mathscr{I}^-} \int_{1+{\bf a_{\it{i}}}}^{1} \frac{{\bf d_{\it{i}}}}{r}\mathcal{H}^{n-1}(rF_i \cap (K\setminus K^{\bf a})) \ dr\\
&\ \ \ \ \ + \sum_{i\in \mathscr{I}^-} \int_{r_i}^{1+{\bf a_{\it{i}}}} \frac{{\bf d_{\it{i}}}}{r}\mathcal{H}^{n-1}(rF_i \cap (K\setminus K^{\bf a})) \ dr+\sum_{i\in \mathscr{I}^+} \int_{r_i}^{1} \frac{{\bf d_{\it{i}}}}{r}\mathcal{H}^{n-1}(rF_i \cap (K\setminus K^{\bf a})) \ dr.
\end{align}
Let $A_1$ denote the latter two terms,
\begin{align}
A_1&:=\sum_{i\in \mathscr{I}^-} \int_{r_i}^{1+{\bf a_{\it{i}}}} \frac{{\bf d_{\it{i}}}}{r}\mathcal{H}^{n-1}(rF_i \cap (K\setminus K^{\bf a})) \ dr+\sum_{i\in \mathscr{I}^+} \int_{r_i}^{1} \frac{{\bf d_{\it{i}}}}{r}\mathcal{H}^{n-1}(rF_i \cap (K\setminus K^{\bf a})) \ dr\\
&=\sum_{i\in \mathscr{I}} \int_{r_i}^{\min\{1,1+{\bf a_{\it{i}}}\}} \frac{{\bf d_{\it{i}}}}{r}\mathcal{H}^{n-1}(rF_i \cap (K\setminus K^{\bf a})) \ dr.
\end{align}
By definition of $r_i$ we have that $rF_i \cap (K\setminus K^{\bf a})\neq \emptyset$ for $r_i<r<\min\{1,1+{\bf a_{\it{i}}}\}$. Hence we can apply Lemma \ref{Transverse Estimate} to estimate
\begin{align}
A_1 \leq C\sum_{i\in \mathscr{I}} \int_{r_i}^{\min\{1,1+{\bf a_{\it{i}}}\}} \frac{{\bf d_{\it{i}}}}{r}\sum_{j \in \mathscr{I}_i^{\bf a}(r)} (r-(1+{\bf a_{\it{j}}})).
\end{align}
Since $j\in \mathscr{I}_i^{\bf a}(r)$ implies that $r>1+{\bf a_{\it{j}}}$ it follows that each term contributes at most for $1+{\bf a_{\it{j}}}<r<\min\{1,1+{\bf a_{\it{i}}}\}$.
Now denote by $\mathscr{I}_i^{\bf a}$ the collection of $j\in \mathscr{I}$ such that $j\in \mathscr{I}_i^{\bf a}(r)$ for some $r\in (1+{\bf a_{\it{j}}},\min\{1,1+{\bf a_{\it{i}}}\})$. We conclude that
\begin{equation}\label{estimate A1}
A_1 \leq  C\sum_{i\in \mathscr{I}}{\bf d_{\it{i}}}\sum_{j\in \mathscr{I}_i^{\bf a}}\int_{1+{\bf a_{\it{j}}}}^{1} \left(1-\frac{1+{\bf a_{\it{j}}}}{r}\right) \ dr = \sum_{i\in \mathscr{I}}{\bf d_{\it{i}}}\sum_{j\in \mathscr{I}_i^{\bf a}}O({\bf a}_j^2) = O(\|{\bf a}\|_{\ell^{\infty}}^2).
\end{equation}
Returning to $A$, we have from \eqref{estimate A1} that
\begin{align}
A &= \sum_{i\in \mathscr{I}^-}\int_{1+{\bf a_{\it{i}}}}^{1} \frac{{\bf d_{\it{i}}}}{r}\mathcal{H}^{n-1}(rF_i \cap (K\setminus K^{\bf a})) \ dr + O(\|{\bf a}\|_{\ell^{\infty}}^2) \\
&= \sum_{i\in \mathscr{I}^-}\int_{1+{\bf a_{\it{i}}}}^{1} {\bf d_{\it{i}}}{\bf m_{\it{i}}}r^{n-2} \ dr + O(\|{\bf a}\|_{\ell^{\infty}}^2)
\end{align}
since again $rF_i \subset K\setminus K^{\bf a}$ for $i\in \mathscr{I}^-$ and $r\in (1+{\bf a_{\it{i}}},1]$. The integral above is $O({\bf a_{\it{i}}})$, so it must be coupled with ${\bf m_{\it{i}}^a}$ in order to use \eqref{delta estimate O(a)}. To rectify this we introduce some additional terms whose error can be well-approximated:
\begin{equation}\label{rewrite A}
A= \sum_{i\in \mathscr{I}^-}\int_{1+{\bf a_{\it{i}}}}^1 \frac{{\bf d_{\it{i}}}{\bf m_{\it{i}}^a}}{r^{n-1}} \ dr + \sum_{i\in \mathscr{I}^-}\int_{1+{\bf a_{\it{i}}}}^1 \left({\bf d_{\it{i}}}{\bf m_{\it{i}}}r^{n-2}-\frac{{\bf d_{\it{i}}}{\bf m_{\it{i}}^a}}{r^{n-1}}\right) \ dr+ O(\|{\bf a}\|_{\ell^{\infty}}^2).
\end{equation}
Let $A_2$ the first sum in  \eqref{rewrite A} and $A_3$ the latter. We estimate $A_2$ as
\begin{equation}
A_2 = \sum_{i\in \mathscr{I}^-}\int_{1+{\bf a_{\it{i}}}}^1 \frac{{\bf d_{\it{i}}}{\bf m_{\it{i}}^a}}{r^{n-1}} \ dr = \sum_{i\in \mathscr{I}^-}\left[-{\bf d_{\it{i}}}{\bf a_{\it{i}}}{\bf m_{\it{i}}^a}+ {\bf m_{\it{i}}^a}O({\bf a}_i^2)\right].
\end{equation}
Next by Corollary \ref{Asymptotics of Parallel Facet Areas} we have ${\bf m_{\it{i}}^a} = {\bf m_{\it{i}}}+O(\|{\bf a}\|_{\ell^{\infty}})$, so that
\begin{equation}\label{estimate A2}
A_2 =  \sum_{i\in \mathscr{I}^-}\left[-{\bf d_{\it{i}}}{\bf a_{\it{i}}}{\bf m_{\it{i}}^a}+ ({\bf m_{\it{i}}}+O(\|{\bf a}\|_{\ell^{\infty}}))O({\bf a}_i^2)\right] = -\sum_{i\in \mathscr{I}^-}{\bf d_{\it{i}}}{\bf a_{\it{i}}}{\bf m_{\it{i}}^a}+ O(\|{\bf a}\|_{\ell^{\infty}}^2).
\end{equation}
We now estimate $A_3$. Adding and subtracting ${\bf d_{\it{i}}}{\bf m_{\it{i}}}/r^{n-1}$, 
\begin{align}
|A_3| &\leq \bigg|\sum_{i\in \mathscr{I}^-}{\bf d_{\it{i}}}\int_{1+{\bf a_{\it{i}}}}^1 \left({\bf m_{\it{i}}}\left(r^{n-2}-\frac{1}{r^{n-1}}\right)+\frac{1}{r^{n-1}}\left({\bf m_{\it{i}}}-{\bf m_{\it{i}}^a}\right)\right) dr\bigg|\\
&\leq \sum_{i\in \mathscr{I}^-}{\bf d_{\it{i}}}{\bf m_{\it{i}}}\int_{1+{\bf a_{\it{i}}}}^1 \left(\frac{1}{r^{n-1}}-r^{n-2}\right) dr+\sum_{i\in \mathscr{I}^-}{\bf d_{\it{i}}}\big|{\bf m_{\it{i}}}-{\bf m_{\it{i}}^a}\big|\int_{1+{\bf a_{\it{i}}}}^1 \frac{1}{r^{n-1}} \ dr.
\end{align}
For the first term we simply have
\begin{equation}\label{estimate A3.1}
\sum_{i\in \mathscr{I}^-}{\bf d_{\it{i}}}{\bf m_{\it{i}}}\int_{1+{\bf a_{\it{i}}}}^1 \left(\frac{1}{r^{n-1}}-r^{n-2}\right) dr = \sum_{i\in \mathscr{I}^-}O({\bf a}_i^2) = O(\|{\bf a}\|_{\ell^{\infty}}^2),
\end{equation}
whereas for the latter we apply Corollary \ref{Asymptotics of Parallel Facet Areas} once more to obtain
\begin{align}
\sum_{i\in \mathscr{I}^-}{\bf d_{\it{i}}}\big|{\bf m_{\it{i}}}-{\bf m_{\it{i}}^a}\big|\int_{1+{\bf a_{\it{i}}}}^1 \frac{1}{r^{n-1}} \ dr &= \sum_{i\in \mathscr{I}^-}O(\|{\bf a}\|_{\ell^{\infty}})\int_{1+{\bf a_{\it{i}}}}^1 \frac{1}{r^{n-1}} \ dr \ \ \ \ \  \\
&= \sum_{i\in \mathscr{I}^-}O(\|{\bf a}\|_{\ell^{\infty}})(-{\bf a_{\it{i}}}+O({\bf a}_i^2)) =  O(\|{\bf a}\|_{\ell^{\infty}}^2), \ \ \ \ \ \ \label{estimate A3.2}
\end{align}
using the fact that ${\bf a_{\it{i}}}<0$. All together, the estimate \eqref{estimate A1} on $A_1$, \eqref{estimate A2} on $A_2$, and \eqref{estimate A3.1} and \eqref{estimate A3.2} on $A_3$ imply that
\begin{equation}
A \leq -\sum_{i\in \mathscr{I}^-}{\bf d_{\it{i}}}{\bf a_{\it{i}}}{\bf m_{\it{i}}^a}+ O(\|{\bf a}\|_{\ell^{\infty}}^2)
\end{equation}
as desired. 

\noindent
\textit{Step 2: Asymptotics of $B$}. Note that $i\in \mathscr{I}^-$ does not contribute to $B$. As mentioned, $rF_i \subset K$ if and only if $0\leq r<1$. So, the only contributions come from $i\in \mathscr{I}^+$, and last at most for $r\in [1,1+{\bf a_{\it{i}}})$ since $K^{\bf a}\subset \{\la x,\nu_i \ra < {\bf d_{\it{i}}}(1+{\bf a_{\it{i}}})\}$. Thus,
\begin{equation}
B = \sum_{i\in \mathscr{I}^+} \int_1^{1+{\bf a_{\it{i}}}} \frac{{\bf d_{\it{i}}}}{r}\mathcal{H}^{n-1}(rF_i \cap (K^{\bf a}\setminus K)) \ dr.
\end{equation}
To estimate $\mathcal{H}^{n-1}(rF_i \cap (K^{\bf a}\setminus K))$ we wish to rewrite it as $\mathcal{H}^{n-1}(F_i^{\bf a_1}\cap F_i ^{\bf a_2})$ and use the fact that
\begin{equation}\label{rewrite intersection estimate}
\mathcal{H}^{n-1}(F_i^{\bf a_1})+\mathcal{H}^{n-1}(F_i^{\bf a_2})-2\mathcal{H}^{n-1}(F_i^{\bf a_1}\cap F_i^{\bf a_2}) = \mathcal{H}^{n-1}(F_i^{\bf a_1}\Delta F_i^{\bf a_2})
\end{equation}
together with Proposition \ref{Symmetric Difference Perturbed Facet Bound}. First we choose ${\bf a_1}$ so that ${\bf a_{1,\it{j}}}=r-1$ for all $j\in \mathscr{I}$. With this, 
\begin{equation}
K^{\bf a_1} = \bigcap_{j\in \mathscr{I}} \{\la x,\nu_j\ra < {\bf d_{\it{j}}}(1+{\bf a_{1,\it{j}}})\} = \bigcap_{j\in \mathscr{I}} \{\la x,\nu_j\ra < r{\bf d_{\it{j}}}\} = rK.
\end{equation}

As for ${\bf a_2}$, notice that for any $r\in [1,1+{\bf a_{\it{i}}})$ we have that $rF_i \cap K = \emptyset$ and therefore $rF_i \cap (K^{\bf a}\setminus K) = rF_i \cap K^{\bf a}$. Next since $rF_i \subset \Sigma_{\nu_i}^r$ we further have $rF_i \cap (K^{\bf a}\setminus K) = rF_i \cap (K^{\bf a}\cap \Sigma_{\nu_i}^r)$. We aim to write $\Cl(K^{\bf a})\cap \Sigma_{\nu_i}^r = F_i^{\bf a_2}$ for some ${\bf a_2}$, which is sufficient since, if this is the case, then $(K^{\bf a}\cap \Sigma_{\nu_i}^r)\Delta F_i^{\bf a_2} = \partial K^{\bf a} \cap \Sigma_{\nu_i}^r$ the latter being at most $n-2$ dimensional.

Consider ${\bf a_2}$ which satisfies ${\bf a_{2,\it{i}}}=r-1$ and ${\bf a_{2,\it{j}}}={\bf a_{\it{j}}}$ for $j\neq i$. By definition
\begin{equation}
F_i^{\bf a_2} := \partial K^{\bf a_2}\cap \Sigma_{\nu_i}^{1+{\bf a_{2,\it{i}}}} = \partial K^{\bf a_2}\cap \Sigma_{\nu_i}^r = \Cl(K^{\bf a_2}) \cap \Sigma_{\nu_i}^r
\end{equation}
where the remaining equalities hold by choice of ${\bf a_2}$ and the fact that $K^{\bf a_2}\subset \{\la x,\nu_i \ra < r{\bf d_{\it{i}}}\}$, implying $K^{\bf a_2}\cap \Sigma_{\nu_i}^r=\emptyset$. On the other hand $x\in F_i^{\bf a_2}$ if and only if $\la x,\nu_i \ra = r{\bf d_{\it{i}}} \leq {\bf d_{\it{i}}}(1+{\bf a_{\it{i}}})$ and $\la x,\nu_j \ra \leq {\bf d_{\it{j}}}(1+{\bf a_{\it{j}}})$ for $j\neq i$. Consequently $F_i^{\bf a_2} = \Cl(K^{\bf a})\cap \Sigma_{\nu_i}^r$.

Invoking Proposition \ref{Symmetric Difference Perturbed Facet Bound}, there exists $C(n,K)>0$ such that 
\begin{equation}\label{sym diff bound cut facet}
\mathcal{H}^{n-1}(F_i^{\bf a_1}\Delta F_i^{\bf a_2})\leq CM_{\Phi}^{n-1}\|{\bf a_1}-{\bf a_2}\|_{\ell^{\infty}}.
\end{equation}
Applying Corollary \ref{Asymptotics of Parallel Facet Areas}
to both $\mathcal{H}^{n-1}(F_i^{\bf a_1})$ and $\mathcal{H}^{n-1}(F_i^{\bf a_2})$, comparing with $\mathcal{H}^{n-1}(F_i^{\bf a})$, we have that
\begin{align}
\mathcal{H}^{n-1}(F_i^{\bf a_1}) &\geq \mathcal{H}^{n-1}(F_i^{\bf a}) - CM_{\Phi}^{n-1}\|{\bf a_1}-{\bf a}\|_{\ell^{\infty}};\label{bound Fia1}\\
\mathcal{H}^{n-1}(F_i^{\bf a_2}) &\geq \mathcal{H}^{n-1}(F_i^{\bf a}) - CM_{\Phi}^{n-1}\|{\bf a_2}-{\bf a}\|_{\ell^{\infty}}. \label{bound Fia2}
\end{align}
Finally notice by our choice of ${\bf a_1}$ and ${\bf a_2}$ that
\begin{equation}
\|{\bf a_1}-{\bf a_2}\|_{\ell^{\infty}} = \max_{j\neq i}|r-1-{\bf a_{\it{j}}}|,  \ \ \ \|{\bf a_1}-{\bf a}\|_{\ell^{\infty}} = \max_{j \in \mathscr{I}}|r-1-{\bf a_{\it{j}}}|, \ \ \ \|{\bf a_2}-{\bf a}\|_{\ell^{\infty}} = |r-1-{\bf a_{\it{i}}}|,
\end{equation}
which are all bounded above by $\max_{j\in \mathscr{I}}|r-1-{\bf a_{\it{j}}}|$. Substituting \eqref{sym diff bound cut facet}, \eqref{bound Fia1}, and \eqref{bound Fia2} with this observation into \eqref{rewrite intersection estimate} yields
\begin{align}
\mathcal{H}^{n-1}(rF_i \cap (K^{\bf a}\setminus K)) &= \mathcal{H}^{n-1}(rF_i \cap (K^{\bf a}\cap \Sigma_{\nu_i}^r)) =\mathcal{H}^{n-1}(F_i^{\bf a_1}\cap F_i^{\bf a_2}) \\
&\geq \frac{1}{2}\left(\mathcal{H}^{n-1}(F_i^{\bf a_1})+\mathcal{H}^{n-1}(F_i^{\bf a_2})-CM_{\Phi}^{n-1}\max_{j\in \mathscr{I}}|r-1-{\bf a_{\it{j}}}|\right)\\
&\geq \mathcal{H}^{n-1}(F_i^{\bf a})-\frac{3}{2}CM_{\Phi}^{n-1}\max_{j\in \mathscr{I}}|r-1-{\bf a_{\it{j}}}|.
\end{align}
Hence, absorbing constants
\begin{align}
-B &=\sum_{i\in \mathscr{I}^+} {\bf d_{\it{i}}}\int_1^{1+{\bf a_{\it{i}}}} -\frac{1}{r}\mathcal{H}^{n-1}(rF_i \cap (K^{\bf a}\setminus K)) \ dr \\
&\leq \sum_{i\in \mathscr{I}^+} {\bf d_{\it{i}}}\int_1^{1+{\bf a_{\it{i}}}} \left(-\frac{{\bf m_{\it{i}}^a}}{r}+C\max_{j\in \mathscr{I}}\left|\frac{1+{\bf a_{\it{j}}}}{r}-1\right|\right) \ dr \leq -\sum_{i\in \mathscr{I}^+} {\bf d_{\it{i}}}{\bf a_{\it{i}}}{\bf m_{\it{i}}^a}+O(\|{\bf a}\|_{\ell^{\infty}}^2), 
\end{align}
where the latter term in the integral is estimated similar to that in \eqref{estimate A1}.  
\end{proof}

\section{Proof of the Main Result}\label{Proof of result}
\subsection{With uniform density estimates} In this section we show that Theorem \ref{Main Result} holds under the assumption that $E$ satisfies uniform density estimates.
\begin{prop}\label{Main Result UDEs} There exists $\epsilon_0(n,K)>0$ and $C(n,K)>0$ such that for any $E\subset \mathbb{R}^n$ satisfying the uniform density estimate \eqref{Uniform density estimate}, $|E|=|K|$, and $|E\Delta K|\leq \epsilon_0$ then we have
$$\beta_{\Phi}(E)^{2} \leq C\,\delta_{\Phi}(E).$$
\end{prop} 
The primary obstacle to proving Proposition \ref{Main Result UDEs} is that we need Lipschitz continuity of $\gamma_{\Phi}$ rather than the H\"{o}lder continuity given by Proposition \ref{Properties}. The goal then is to prove a statement about the Lipschitz continuity of $\gamma_{\Phi}$ at the parallel polytope $K^{\bf a*}$ obtained from Theorem \ref{Replacement}. 
\begin{prop}\label{Vanishing Gradient} There exists $\epsilon_0(n,K)>0$ such that for any  $E\subset \mathbb{R}^n$ satisfying the uniform density estimate \eqref{Uniform density estimate}, $|E|=|K|$, and $|E\Delta K|\leq \epsilon_0$ then we have
$$|\gamma_{\Phi}(E)-\gamma_{\Phi}(K^{{\bf a}*})| \leq \frac{1}{2}|E\Delta K^{{\bf a}*}|.$$
\end{prop}

We begin with the following essential remark, which shows that in fact the polytope $K^{\bf a*}$ is unique and justifies our labeling.
\begin{rem}\label{Properties of Ka*}
Returning to Theorem \ref{Replacement}, the set $K^{\bf a}$ obtained is actually the unique such $K^{\bf a}$ satisfying that $\mathcal{H}^{n-1}(\partial E \cap \Cone(F_i^{\bf a})) = \mathcal{H}^{n-1}(\partial K^{\bf a} \cap \Cone(F_i^{\bf a}))$ for each $i\in \mathscr{I}$; call it $K^{{\bf a}*}$. In some sense, $K^{{\bf a}*}$ is like a projection of $E$ onto a close polytope to $K$. Moreover, as discussed in \cite[Lemma 2.2]{FZ2019}, we actually have $|{\bf a*}|\leq C(n,K)|E\Delta K|$.  Note then that by Lemma \ref{Equivalence} we have
\begin{equation}\label{Parallel L1 close}
|K^{{\bf a}*}\Delta K| \leq C(n,K)\|{\bf a*}\|_{\ell^{\infty}} \leq C(n,K)|E\Delta K|
\end{equation}
where $C(n,K)>0$ changes throughout.
\end{rem}
We now show that if $E$ is $L^1$ close to $K$ and satisfies the uniform density estimates, then actually both $E$ and $K^{\bf a*}$ are close in $d_H$ to $K$. 
\begin{lem}\label{Closeness} There exists $C(n,K)>0$ such that if $E\subset \mathbb{R}^n$ satisfies the uniform density estimate \eqref{Uniform density estimate} then
\begin{equation}\label{hausdorff closeness E}
d_H(\partial E,\partial K)^n \leq C\, |E\Delta K|.
\end{equation}
and if in addition $E$ satisfies the hypotheses of Proposition \ref{Vanishing Gradient},
\begin{equation}\label{hausdorff closeness Ka*}
d_H(\partial K^{{\bf a}*}, \partial K)\leq C|E\Delta K|.
\end{equation}
\end{lem}
\begin{proof} The content of \eqref{hausdorff closeness E} is precisely \cite[Lemma 3.4]{Neumayer2016}. It follows from a standard argument estimating $|E\Delta K|$ from below in terms of either $|B_d(x) \cap E|$ or $|B_d(x)\setminus E|$, where $d:=d_H(\partial E, \partial K)$ and $x\in \partial^* E$, since $B_d(x)$ is either contained entirely in $K^c$ or in $K$ respectively,. 

As for \eqref{hausdorff closeness Ka*}, by Remark \ref{Properties of Ka*} there exists $C_1(n,K)>0$ such that
$$|K^{{\bf a}*}\Delta K| \leq C_1|E\Delta K|.$$
By Lemma \ref{Equivalence} we additionally find $C_2(n,K)>0$ such that
\begin{equation}
d_H(K^{\bf a*},K) \leq C_2|K^{\bf a*}\Delta K|.
\end{equation}
In turn,
$$d_H(\partial K^{{\bf a}*},\partial K)=d_H(K^{{\bf a}*},K) \leq C_2|K^{{\bf a}*}\Delta K|\leq C_1C_2|E\Delta K|.$$
The first equality holds for any two convex bodies, see for example \cite[Theorem 20]{Wil07}.
\end{proof}
From Lemma \ref{Closeness} we are ready to prove Proposition \ref{Vanishing Gradient}.
\begin{proof}[Proof of Proposition \ref{Vanishing Gradient}] We prove something more general. Let $0<\epsilon<1/2$ and choose $\epsilon_0< \eta$ as in Lemma \ref{Center norm}. Up to decreasing $\epsilon_0$ (requiring it to now also depend on $n$ and $K$), the bounds 
\eqref{hausdorff closeness E} and \eqref{hausdorff closeness Ka*} on the Hausdorff distances of $E$ and $K^{\bf a*}$ to $K$ respectively imply that
\begin{equation}\label{Containment}
(1-\epsilon)K\subset E, K^{{\bf a}*} \subset (1+\epsilon)K.
\end{equation}
Next, by the control on the center from Lemma \ref{Center norm} we have $|y_E|<\epsilon$. However, by Remark \ref{Properties of Ka*}, up to further decreasing $\epsilon_0$, we can also apply Lemma \ref{Center norm} to $K^{\bf a*}$ and conclude that $|y_{K^{{\bf a}*}}|<\epsilon$ too. 

Recall that $f_*(x) = \sup\{\la x,\nu \ra \ | \ f(\nu)\leq 1\}$. We estimate $f_*(x-y)$ for any $x,y \in \mathbb{R}^n$ by a direct application of Cauchy-Schwarz to obtain
\begin{equation}\label{Gauge bound}
f_*(x)-|y|\leq f_*(x-y)\leq f_*(x)+|y|.
\end{equation}
Now, by \eqref{Containment} we see that if $x\in E\cup K^{{\bf a}*}$ then $f_*(x) \leq 1+\epsilon$ owing to the 1-homogeneity of $f_*$. Similarly, if $x\in (E\cap K^{{\bf a}*})^c$ then $f_*(x)\geq 1-\epsilon$. Accordingly, for all $x\in E\Delta K^{{\bf a}*}$,
\begin{equation}\label{Gauge bound eps}
(1-\epsilon)-|y|\leq f_*(x-y)\leq (1+\epsilon)+|y|.
\end{equation}
If $\gamma_{\Phi}(E)\leq\gamma_{\Phi}(K^{{\bf a}*})$ then, choosing $y=y_{K^{{\bf a}*}}$ we get 
\begin{align*}
\gamma_{\Phi}(K^{{\bf a}*})-\gamma_{\Phi}(E) &\leq \int_{K^{{\bf a}*}} \frac{1}{f_*(x-y_{K^{{\bf a}*}})} \ dx - \int_{E} \frac{1}{f_*(x-y_{K^{{\bf a}*}})} \ dx \\
&= \int_{K^{{\bf a}*}\setminus E} \frac{1}{f_*(x-y_{K^{{\bf a}*}})} \ dx - \int_{E\setminus K^{{\bf a}*}} \frac{1}{f_*(x-y_{K^{{\bf a}*}})} \ dx \\
& \leq \int_{K^{{\bf a}*}\setminus E} \frac{1}{1-2\epsilon} \ dx - \int_{E\setminus K^{{\bf a}*}} \frac{1}{1+2\epsilon} \ dx \leq 2\epsilon|E\Delta K^{{\bf a}*}|,
\end{align*}
where we have applied \eqref{Gauge bound eps} in both directions. The same inequality holds if $\gamma_{\Phi}(K^{{\bf a}*})\leq \gamma_{\Phi}(E)$ using $y_E$; it is for this reason that we need control on the containment of both $E$ and $K^{{\bf a}*}$. Thus,
$$|\gamma_{\Phi}(E)-\gamma_{\Phi}(K^{{\bf a}*})|<2\epsilon |E\Delta K^{{\bf a}*}|$$
and setting $\epsilon=1/4$ gives the result.
\end{proof}

We are finally ready to prove Proposition \ref{Main Result UDEs}.
\begin{proof} We use the properties of $\beta_{\Phi}$ and $\gamma_{\Phi}$, together with Theorem \ref{Replacement} and the Lipschitz continuity of $\gamma_{\Phi}$ from Proposition \ref{Vanishing Gradient} to conclude the result.

Let $\epsilon_0$ be small enough to invoke both Proposition \ref{Vanishing Gradient} and Theorem \ref{Replacement}. Then, by the latter, for some $C_1(n,K)>0$
\begin{equation}\label{Perimeter lower bound}
\Phi(E)-\Phi(K^{{\bf a}*}) \geq C_1 |E\Delta K^{{\bf a}*}|.
\end{equation}
By definition of $\delta_{\Phi}$ we can rewrite \eqref{Perimeter lower bound} as 
\begin{equation}\label{Deficit strong bound}
\frac{C_1}{n|K|}|E\Delta K^{{\bf a}*}| \leq \frac{\Phi(E)}{n|K|}-1+1-\frac{\Phi(K^{{\bf a}*})}{n|K|} = \delta_{\Phi}(E)-\delta_{\Phi}(K^{{\bf a}*})
\end{equation}
as $|E|=|K^{{\bf a}*}|=|K|$. Notice immediately \eqref{Deficit strong bound} implies
\begin{equation}\label{Deficit bound}
\delta_{\Phi}(K^{{\bf a}*}) \leq \delta_{\Phi}(E).
\end{equation} 
Now let us turn to estimating $\beta_{\Phi}(E)^2$. Again by definition,
\begin{equation}\label{beta defn}
\beta_{\Phi}(E)^2 = \frac{\Phi(E)}{n|K|}-\frac{(n-1)\gamma_{\Phi}(E)}{n|K|}.
\end{equation}
To handle the $\gamma_{\Phi}(E)$ term we appeal to the Lipschitz continuity of $\gamma_{\Phi}$ in Proposition \ref{Vanishing Gradient} to find
\begin{equation}\label{Lipschitz estimate}
|\gamma_{\Phi}(K^{{\bf a}*})-\gamma_{\Phi}(E)| \leq \frac{1}{2}|E\Delta K^{{\bf a}*}|.
\end{equation}
Rearranging \eqref{Lipschitz estimate} and substituting in \eqref{beta defn} gives
\begin{equation}\label{Oscillation bound}
\beta_{\Phi}(E)^2 \leq \frac{\Phi(E)}{n|K|}-\frac{(n-1)\gamma_{\Phi}(K^{{\bf a}*})}{n|K|} + \frac{n-1}{2n|K|}|E\Delta K^{{\bf a}*}|.
\end{equation}
We must bound each term on the right-hand side of \eqref{Oscillation bound} by $\delta_{\Phi}(E)$. We start with the first two terms. By adding and subtracting $\Phi(K^{{\bf a}*})/(n|K|)$ we have
\begin{align}
\frac{\Phi(E)}{n|K|}-\frac{(n-1)\gamma_{\Phi}(K^{{\bf a}*})}{n|K|} &= \frac{\Phi(E)}{n|K|}-\frac{\Phi(K^{{\bf a}*})}{n|K|}+\beta_{\Phi}(K^{{\bf a}*})^2 \nonumber \\
&=\delta_{\Phi}(E)-\delta_{\Phi}(K^{{\bf a}*})+\beta_{\Phi}(K^{{\bf a}*})^2.\label{First two terms estimation}
\end{align}
Since $K^{{\bf a}*}$ is parallel to $K$, up to translating by some $x^{\bf a*}$ we can apply Proposition \ref{Result for Parallel Polytopes} and find $C_2(n,K)$ such that $\beta_{\Phi}(K^{{\bf a}*}+x^{\bf a*})^2 \leq C_2\delta_{\Phi}(K^{{\bf a}*}+x^{\bf a^*})$. However, both $\beta_{\Phi}$ and $\delta_{\Phi}$ are translation invariant as seen in Remark \ref{Invariance}. Consequently, $\beta_{\Phi}(K^{\bf a*})^2\leq C_2\delta_{\Phi}(K^{\bf a*})$. Applying this together with \eqref{First two terms estimation} in \eqref{Oscillation bound} yields
\begin{equation}
\beta_{\Phi}(E)^2 \leq \delta_{\Phi}(E)+(C_2-1)\delta_{\Phi}(K^{{\bf a}*}) + \frac{n-1}{2n|K|}|E\Delta K^{{\bf a}*}|.
\end{equation}
As for the remaining term, we directly apply our bound  \eqref{Deficit strong bound} on $|E\Delta K^{\bf a*}|$ to conclude
$$\beta_{\Phi}(E)^2 \leq \left(1+\frac{(n-1)}{2C_1}\right)\delta_{\Phi}(E)+\left(\frac{(n-1)}{2C_1}+C_2-1\right)\delta_{\Phi}(K^{{\bf a}*}).$$ 
Letting $C_3 = \max\{0,(n-1)/(2C_1)+C_2-1\}$ we apply $\delta_{\Phi}(K^{\bf a*})\leq \delta_{\Phi}(E)$ from \eqref{Deficit bound} to conclude
$$\beta_{\Phi}(E)^2 \leq \left(1+C_3+\frac{(n-1)}{2C_1}\right)\delta_{\Phi}(E).$$ 
So setting 
$$C=1+\max\left\lbrace 0,\frac{(n-1)}{2C_1(n,K)}+C_2(n,K)-1\right\rbrace + \frac{(n-1)}{2C_1(n,K)},$$
which evidently depends only on $n$ and $K$ and is positive, the proof is complete.
\end{proof}
\subsection{Proof of Theorem \ref{Main Result}}
We are finally ready to prove the main theorem: 
\setcounter{section}{1}
\setcounter{thm}{0}
\begin{thm} There exists $C(n,K)>0$ such that for any $E\subset \mathbb{R}^n$ a set of finite perimeter with $0<|E|<\infty$,
$$\alpha_{\Phi}(E)^2+\beta_{\Phi}(E)^2 \leq C\delta_{\Phi}(E).$$
\end{thm}
\setcounter{section}{4}
\setcounter{thm}{7}
Following \cite{Neumayer2016}, we use the same selection principle argument. The proof essentially follows that in \cite[Theorem 1.1]{Neumayer2016}; we include it here for the sake of completeness.
\begin{proof}
By \eqref{anisotropic Isop. ineq} it suffices to show there exists $C(n,K)>0$ such that
\begin{equation}\label{Final}
\beta_{\Phi}(E)^2 \leq C\delta_{\Phi}(E).
\end{equation}
Moreover, by Lemma \ref{Ball containment} we may assume $E\subset B_{R_0}$. 

\noindent
\textit{Step 1: Generating a convergent sequence.} We will show Theorem \ref{Main Result} with a proof by contradiction. The first step is to take a minimizing sequence along which the claim is supposedly false and show it converges to (a translation of) $K$.

Suppose \eqref{Final} is false. Then we can find a sequence of sets $\{E_j\}_{j=1}^{\infty}$ in $B_{R_0}$ such that $|E_j|=|K|$, $\Phi(E_j)\to \Phi(K)$, but 
\begin{equation}\label{Final Contradiction}
\Phi(E_j)< \Phi(K)+C_2\beta_{\Phi}(E_j)^2
\end{equation}
for $C_2>0$ to be chosen later. The sequence $\{E_j\}_{j=1}^{\infty}$ will be replaced with a new one $\{F_j\}_{j=1}^{\infty}$ with more regularity to exploit. Since $E_j\subset B_{R_0}$ for all $j$ and $P(E)\leq \Phi(E)/m_{\Phi}<\infty$,
we have by compactness, see \cite[Theorem 12.26]{Maggi}, that $E_j\xto{L^1} E_{\infty}$ up to subsequence. Since $|E_j|=|K|$ for all $j$ we also have $|E_{\infty}|=|K|$. By \eqref{anisotropic Isop. ineq} and Proposition \ref{Properties} part ii),
$$n|K|\leq \Phi(E_{\infty}) \leq \liminf_{j\to \infty}\Phi(E_j) = \Phi(K)=n|K|,$$
so that $E_{\infty}=K$ up to translation. Futhermore, by Proposition \ref{Properties} parts i) and ii) and \eqref{gamma},
$$\lim_{j\to \infty} \gamma_{\Phi}(E_j) = \gamma_{\Phi}(K) =\frac{n|K|}{n-1}, \ \ \ \lim_{j\to \infty}\beta_{\Phi}(E_j) = \lim_{j\to \infty}\left[\frac{\Phi(E_j)}{n|K|} - \frac{n-1}{n|K|}\gamma_{\Phi}(E_j)\right]=0.$$
\noindent
\textit{Step 2: Replacing $E_j$ with almost-minimizers of $\Phi$.} We replace the sequence $\{E_j\}_{j=1}^{\infty}$ with a new sequence $\{F_j\}_{j=1}^{\infty}$ which minimize $Q_j$ among sets $E\subset B_{R_0}$, where
\begin{equation}
Q_j(E) := \Phi(E)+\frac{\epsilon_{\Phi}|K|}{8}|\beta_{\Phi}(E)^2-\beta_{\Phi}(E_j)^2| + \Lambda\big||E|-|K|\big|
\end{equation}
and $\Lambda>4n$ is fixed. Since the functionals $Q_j$ are slight modifications of the (volume-constrained) anisotropic perimeter, the $F_j$ are almost minimizers. The regularity theory of almost minimizers then allows us to show that for some $R_0>0$ 
\begin{enumerate}[label = \roman*)]
\item The following bounds hold for $\beta_{\Phi}(F_j)$ and $\Phi(F_j)$:
\begin{equation}\label{Bound Contradiction}
\beta_{\Phi}(E_j)^2\leq 2\beta_{\Phi}(F_j)^2 \ \ \ \text{and} \ \ \ \Phi(F_j)\leq \Phi(E_j).
\end{equation}
In particular, with the contradictory hypothesis \eqref{Final Contradiction},
\begin{equation}\label{Replaced Contradiction}
\Phi(F_j)< \Phi(K)+2C_2\beta_{\Phi}(F_j)^2.
\end{equation}
\item $|F_j|=|K|$ for $j$ sufficiently large, and
\item $F_j$ converges in $L^1$ to $K$,
\end{enumerate}
Items i) and ii) allow us to replace $E_j$ with $F_j$ and item iii) recovers the same information obtained in Step 1. 

To show \eqref{Bound Contradiction}, observe first that Lemma \ref{Minimize Q} guarantees a minimizer of $Q_j$ among $E\subset B_{R_0}$ exists -- let $F_j\in \argmin\{Q_j(E) \ | \ E\subset B_{R_0}\}$. Since $F_j$ minimizes $Q_j$, testing with $E_j$ yields
\begin{equation}\label{test Qj with Ej}
\Phi(F_j)+ \frac{\epsilon_{\Phi}|K|}{8}|\beta_{\Phi}(F_j)^2-\beta_{\Phi}(E_j)^2|+ \Lambda \big||F_j|-|K|\big|\leq \Phi(E_j)
\end{equation}
where we have used the fact that $Q_j(E_j)=\Phi(E_j)$ by construction. In particular, $\Phi(F_j)\leq \Phi(E_j)$. As for the bound $\beta_{\Phi}(E_j)^2\leq 2\beta_{\Phi}(F_j)^2$, combining \eqref{test Qj with Ej} with \eqref{Final Contradiction} we get
\begin{equation}\label{test with contradiction hyp.}
\Phi(F_j)+ \frac{\epsilon_{\Phi}|K|}{8}|\beta_{\Phi}(F_j)^2-\beta_{\Phi}(E_j)^2|+ \Lambda \big||F_j|-|K|\big|\leq \Phi(E_j)\leq \Phi(K)+C_2\beta_{\Phi}(E_j)^2.
\end{equation}
On the other hand, recall Lemma \ref{Unique Min}, which quantifies the fact that $K$ minimizes the volume-constrained minimization of $\Phi$. This lets us bound $\Phi(K)$ in terms of $F_j$ in \eqref{test with contradiction hyp.} and gives
\begin{equation}
\Phi(F_j)+ \frac{\epsilon_{\Phi}|K|}{8}|\beta_{\Phi}(F_j)^2-\beta_{\Phi}(E_j)^2|+ \Lambda \big||F_j|-|K|\big|\leq \Phi(F_j) + \Lambda \big||F_j|-|K|\big| + C_2\beta_{\Phi}(E_j)^2.
\end{equation}
It immediately follows that
\begin{equation}\label{beta bound contradiction}
\frac{\epsilon_{\Phi}|K|}{8}|\beta_{\Phi}(F_j)^2-\beta_{\Phi}(E_j)^2|\leq C_2\beta_{\Phi}(E_j)^2.
\end{equation}
From Step 1 we have $\beta_{\Phi}(E_j)\to 0$, and so $\beta_{\Phi}(F_j)\to 0$ too. Next, rearranging \eqref{beta bound contradiction} and choosing $C_2(n,K)$ sufficiently small,
\begin{equation}\label{beta lower bound contradiction}\frac{\beta_{\Phi}(E_j)^2}{2}\leq\beta_{\Phi}(E_j)^2 - \frac{8C_2}{\epsilon_{\Phi}|K|}\beta_{\Phi}(E_j)^2 \leq \beta_{\Phi}(F_j)^2, \ \ \ \text{i.e.} \ \ \beta_{\Phi}(E_j)^2 \leq 2\beta_{\Phi}(F_j)^2.
\end{equation}
This completes the proof of i). 

We now show item ii). Let $r_j>0$ be such that $|r_jF_j|=|K|$. By \eqref{Minimize Q Properties} we see that $r_j^n|K|/2\leq r_j^n|F_j|=|K|$, and so $r_j^n \leq 2$. By Lemma \ref{Minimize Q Uniform} and Lemma \ref{Closeness}, for $j$ sufficiently large we have $F_j \subset 2K$. Combined with $r_j\leq 2$, we see $2F_j \subset 4K$. Then, as long as $R_0>4M_{\Phi}$ we have $r_jF_j\subset 4K\subset B_{R_0}$. This allows us to test $r_jF_j$ against $F_j$ in $Q_j$ and obtain
\begin{align}
&\Phi(F_j)+ \frac{\epsilon_{\Phi}|K|}{8}|\beta_{\Phi}(F_j)^2-\beta_{\Phi}(E_j)^2|+ \Lambda \big||F_j|-|K|\big|\\
&\leq \Phi(r_jF_j)+ \frac{\epsilon_{\Phi}|K|}{8}|\beta_{\Phi}(r_jF_j)^2-\beta_{\Phi}(E_j)^2|= r^{n-1}\Phi(F_j)+ \frac{\epsilon_{\Phi}|K|}{8}|\beta_{\Phi}(F_j)^2-\beta_{\Phi}(E_j)^2|
\end{align}
owing to the fact that $|r_jF_j|=|K|$ and the scale invariance of $\beta_{\Phi}$ (see Remark \ref{Invariance}). Thus,
\begin{equation}\label{Anisotropic bound contradiction}
\Phi(F_j) +\Lambda\big||F_j|-|K|\big| \leq r_j^{n-1}\Phi(F_j).
\end{equation}
Since the latter term on the left-hand side of \eqref{Anisotropic bound contradiction} is non-negative, we see that $r_j\geq 1$ for all $j$. Moreover, as $F_j\xto{L^1}K$ we have $|F_j|\to|K|$, i.e. $r_j\to 1$. Suppose that, up to a subsequence, $r_j>1$. Since $|F_j|=|K|/r_j^n$, rearranging \eqref{Anisotropic bound contradiction} gives
\begin{equation}\label{Lambda bound contradiction}
\Lambda|K| \leq \left(\frac{r_j^n(r_j^{n-1}-1)}{r_j^n-1}\right)\Phi(F_j).
\end{equation}
Next, since $r_j^n(r_j^{n-1}-1)/(r_j^n-1)\to (n-1)/n$ as $r_j\to 1^+$ we have for any $0<\epsilon<1/n$ and $j$ sufficiently large that the right-hand side of \eqref{Lambda bound contradiction} is bounded by $(1-\epsilon)\Phi(F_j)$. Finally, testing $F_j$ against $K$ in $Q_j$ gives
\begin{equation}\label{Anisotropic bound 2 contradiction}
\Phi(F_j) \leq Q_j(F_j) \leq Q_j(K) = n|K|+\frac{\epsilon_{\Phi}|K|}{8}\beta_{\Phi}(E_j)^2.
\end{equation}
Combining \eqref{Lambda bound contradiction} with \eqref{Anisotropic bound 2 contradiction} for $j$ sufficiently large and $\beta_{\Phi}(E_j)\to 0$,
$$\Lambda|K| \leq (1-\epsilon)\Phi(F_j) \leq (1-\epsilon)\left(n|K|+\frac{\epsilon_{\Phi}|K|}{8}\beta_{\Phi}(E_j)^2\right) \leq n|K|.$$
However, we assumed that $\Lambda>4n$, a contradiction to the supposition that $r_j>1$. 

To show iii) we need a uniform upper bound on $P(F_j)$. From \eqref{Minimize Q Properties} we can control $\Phi(F_j)$ in terms of $\Phi(K)$, and so
$$P(F_j) \leq \frac{1}{m_{\Phi}}\Phi(F_j) \leq \frac{2\Phi(K)}{m_{\Phi}}<\infty.$$
Thus up to subsequence $F_j\xto{L^1}F_{\infty}$ for some $F_{\infty}\subset B_{R_0}$. Appealing to \eqref{test with contradiction hyp.}, since $\beta_{\Phi}(E_j)\to 0$,
\begin{align}
\Phi(F_{\infty})\hspace{-2pt}+\hspace{-2pt}\Lambda\liminf_{j\to \infty}\big||F_j|-|K|\big|&\leq \liminf_{j\to \infty}\left[\Phi(F_j)\hspace{-1pt}+\hspace{-1pt}\frac{\epsilon_{\Phi}|K|}{8}|\beta_{\Phi}(F_j)-\beta_{\Phi}(E_j)|^2 \hspace{-2pt}+\hspace{-2pt}\Lambda\big||F_j|-|K|\big|\right] \\
&\leq \liminf_{j\to\infty}\left[\Phi(K)+C_2\beta_{\Phi}(E_j)^2\right] = \Phi(K),
\end{align}
which implies that $F_{\infty}=K$ up to translation by \eqref{anisotropic Isop. ineq}. Without loss of generality, since all the involved quantities are translation invariant (see Remark \ref{Invariance}), we may translate each $F_j$ so that $|F_j\Delta K| = \inf\{|F_j\Delta (K+x) \ | \ x\in\mathbb{R}^n\}$. Thus, $F_j\xto{L^1}K$.

\noindent
\textit{Step 3: Deriving a contradiction to \eqref{Final Contradiction}.} We now use the properties of $F_j$ established in Step 2 to show that $F_j=K$ for $j$ sufficiently large. Since $\beta_{\Phi}(F_j)$ controls $\beta_{\Phi}(E_j)$, and we assumed that $\beta_{\Phi}(E_j)>0$ in \eqref{Final Contradiction}, this gives us a contradiction.

Recall that the $F_j$ minimize $Q_j$, and so by Lemma \ref{Minimize Q Uniform} they satisfy the uniform density estimates. Invoking Proposition \ref{Main Result UDEs}, we obtain $\epsilon_0(n,K)>0$ and $C(n,K)>0$, so that for $j$ sufficiently large as to ensure $|F_j \Delta K| \leq \epsilon_0$, we have
\begin{equation}\label{Oscillation Deficit Contradiction Bound}
\beta_{\Phi}(F_j)^2\leq C_1(n,K)\delta_{\Phi}(F_j).
\end{equation}
Since the $F_j$ minimize $Q_j$, we may test against $E_j$ and see that
$$\Phi(F_j) \leq \Phi(F_j) + \frac{\epsilon_{\Phi}|K|}{8}|\beta_{\Phi}(F_j)^2-\beta_{\Phi}(E_j)^2|\leq \Phi(E_j).$$
Recalling \eqref{Replaced Contradiction} from Step 2 item i), we can replace $F_j$ with $E_j$:
$$\Phi(F_j) \leq \Phi(E_j) \leq \Phi(K)+C_2\beta_{\Phi}(E_j)^2\leq \Phi(K)+2C_2\beta_{\Phi}(F_j)^2.$$
Finally, we use the control on the oscillation index by the deficit established in \eqref{Oscillation Deficit Contradiction Bound} to yield
$$\Phi(F_j)\leq \Phi(K)+2C_2C_1(n,K)\delta_{\Phi}(F_j).$$
Letting $C_2=n|K|/(2C_1)$ we see
$$\Phi(F_j)\leq \Phi(K)+2C_1C_2\delta_{\Phi}(F_j)= \Phi(K)+n|K|\left(\frac{\Phi(F_j)}{n|K|}-1\right) = \Phi(F_j),$$
implying that $F_j=K$ for $j$ sufficiently large (since we translated the $F_j$ to minimize the $L^1$ distance to $K$). However, this means that $\beta_{\Phi}(F_j)=0$ for $j$ sufficiently large, a contradiction to \eqref{Bound Contradiction} as $\beta_{\Phi}(E_j)>0$. 
\end{proof}

\bibliography{Bibliography}
\bibliographystyle{alpha}
\end{document}